\documentclass[11pt]{amsart}

\usepackage{amsfonts, amsmath, amssymb, amsgen, amsthm, amscd}

\usepackage[utf8]{inputenc} 

\usepackage{color}

\usepackage[all]{xy}

\usepackage[top=2.75cm, left=2.4cm, bottom=2.4cm, right=2.4cm]{geometry}

\usepackage{enumerate}

\def\Nb{{\mathbb N}}

\def\a{\alpha}
\def\b{\beta}
\def\d{\delta}
\def\D{\Delta}

\def\om{\omega}

\def\s{\sigma}

\def\t{\theta}

\def\ve{\varepsilon}
\def\vp{\varphi}

\def\Id{\mathop{\rm Id}\nolimits}
\def\p{\partial}
\def\lra{\longrightarrow}

\def\ot{\otimes}
\def\odots{\ot\cdots\ot}

\def\nb{\nabla}
\def\lra{\longrightarrow}
\def\wtilde{\widetilde}

\def\rt{\triangleright}

\def\D{\Delta}

\def\Id{\mathop{\rm Id}\nolimits}

\def\Ad{\mathop{\rm Ad}\nolimits}

\newcommand{\ps}[1]{~\hspace{-4pt}_{^{(#1)}}}

\newcommand{\ns}[1]{~\hspace{-4pt}_{_{{<#1>}}}}

\usepackage{enumerate}

\newcommand{\G}[1]{\mathfrak{#1}}
\newcommand{\C}[1]{\mathcal{#1}}
\newcommand{\B}[1]{\mathbb{#1}}
\newcommand{\lmod}[1]{{#1}\text{-{\bf Mod}}}
\newcommand{\wbar}[1]{\overline{#1}}

\newcommand{\xla}[1]{\xleftarrow{\ #1\ }}

\newcommand{\CB}{{\rm CB}}
\newcommand{\CH}{{\rm CH}}
\newcommand{\CC}{{\rm CC}}
\newcommand{\CDD}{{\rm CD}}
\newcommand{\Hom}{{\rm Hom}}
\newcommand{\tor}{{\rm Tor}}

\newcommand{\ext}{{\rm Ext}}
\newcommand{\ie}{{\it i.e.\/}\ }

\newcommand{\tr}{\triangleright}
\newcommand{\tl}{\triangleleft}

\renewcommand{\leq}{\leqslant}
\renewcommand{\geq}{\geqslant}

\numberwithin{equation}{section}

\newtheorem{theorem}{Theorem}[section]
\newtheorem{proposition}[theorem]{Proposition}
\newtheorem{lemma}[theorem]{Lemma}
\newtheorem{corollary}[theorem]{Corollary}
\theoremstyle{definition}

\newtheorem{definition}[theorem]{Definition}

\title{Hopf-dihedral (co)homology and $L$-theory}

\author{A. Kaygun}
\address{Istanbul Technical University, Istanbul, Turkey}
\email{atabey.kaygun@gmail.com}

\author{S. Sütlü}
\address{Işık University, Istanbul, Turkey}
\email{serkan.sutlu@isikun.edu.tr}

\begin{document}

\maketitle

\begin{abstract}
  We develop an appropriate dihedral extension of the Connes-Moscovici
  characteristic map for Hopf *-algebras.  We then observe that one
  can use this extension together with the dihedral Chern character to
  detect non-trivial $L$-theory classes of a *-algebra that carry a
  Hopf symmetry over a Hopf *-algebra.  Using our machinery we detect
  a previously unknown $L$-class of the standard Podleś sphere.
\end{abstract}

\section*{Introduction}

In this paper we calculate a new class of finer homological invariants
of noncommutative spaces using their Hopf-algebraic symmetries.  We
borrow our strategy from the study of characteristic classes of
topological manifolds where group symmetries of vector or fiber
bundles on a manifold are used to obtain topological invariants of the
underlying manifold.  In noncommutative geometry one can similarly
obtain topological invariants of noncommutative spaces by using Hopf
symmetries of the underlying space and a canonical characteristic map
relating cohomological invariants of the underlying space and its Hopf
algebra symmetries.  One can see a beautiful and effective execution
of this strategy by Connes and Moscovici in their calculation of
topological characteristic classes of codimension-1 foliations using
the Hopf algebra $\C{H}_1$ and the characteristic map they
constructed~\cite{ConnMosc98,ConnMosc}.  Our main contribution in this
paper is a new cohomological tool particularly designed to detect
finer (more geometric) invariants of noncommutative spaces.

Cohomological invariants have been used to obtain topological
invariants of spaces effectively in the
past~\cite{Atiyah:KTheoryLectures, MilnorAndStasheff:KTheory,
  MilnorAndStasheff:CharacteristicClasses}.  We observe that geometric
invariants require much finer cohomological machineries.  From the
point of view of noncommutative geometry, this means one must study
algebras with additional structures, and use cohomology theories
sensitive to these additional structures.  Our first candidate is an
obvious, but cohomologically underused one: algebras and Hopf-algebras
that carry involutive endo-anti-morphisms called $\ast$-structures,
and their Hopf symmetries compatible with these $\ast$-structures.

The dominant cohomological machinery used in noncommutative geometry
appears to be the cyclic cohomology in many
manifestations~\cite{Conn83, Tsyg83, Karo83-III, Conn85, Loday-book,
  CuntQuil95}.  This is mainly because of the Chern character that can
detect non-trivial $K$-invariants of the underlying noncommutative
space using cyclic cocycles~\cite{Conn85,Connes-book}.  However,
cyclic cohomology is but one cohomology theory among a collection of
similarly defined other theories.  Each of these theories is indexed
by collections of groups including all finite dihedral groups, finite
symmetric groups, finite hyperoctahedral groups, and their
corresponding artin groups~\cite{Aboughazi:CrossedSimplicialGroups,
  Loday91}. In the case of cyclic cohomology, we consider finite
cyclic groups of all orders.  Our choice of dihedral cohomology as the
finer replacement of cyclic cohomology is dictated by the fact that
combining finite cyclic groups with an involution yield finite
dihedral groups.

By introducing a $*$-structure on the underlying algebra $A$, we jump
from the realm of $K$-theory into the realm of $L$-theory
\cite{Wall-book, Ranicki92}.  Since the natural extension of the Chern
character to $L$-theory uses dihedral
cocycles~\cite{KrasLapiSolo87,Cortinas:1993a,Cortinas:1993b}, a new
strategy emerges for detecting $L$-classes using a Hopf-dihedral
variant of our characteristic map.  Using this strategy we managed to
detect a previously unknown $L$-class of the standard Podleś sphere
$\C{O}(S_{q}^2)$ in Section~\ref{Sect6}.  From this point of view,
using dihedral cohomology of Hopf $\ast$-algebras and the associated
characteristic map as a refinement of the existing machinery for
Hopf-cyclic cohomology is justified since $L$-theory already yields
much finer invariants than $K$-theory.

\subsection*{Outline of the paper}

We start by defining cyclic and dihedral cohomologies as derived
functors of diagrams of vector spaces given over the cyclic and
dihedral categories in Section~\ref{Sect1}.  Then we recall the
definitions of the cyclic and the dihedral cohomologies of
$\ast$-algebras and $\ast$-coalgebras in Section~\ref{Sect2}.  In
Section~\ref{Sect3} we develop the Hopf-dihedral cohomology, with a
characteristic function built-in, of a Hopf $\ast$-algebra that has a
modular pair in involution (MPI) and an invariant trace.  We are going
to use this characteristic map, as we described above, to detect
$L$-theory classes of an algebra that carries a specific Hopf
symmetry.  In the same section we then extend the theory from MPI to
arbitrary stable coefficients, and investigate the natural
multiplicative structure on the Hopf-Hochschild cohomology, and the
*-structure induced on the cohomology.  In Section~\ref{Sect4} we
consider the complexified quantum enveloping algebra $\G{U}_q(\G{g})$
of a Lie algebra $\G{g}$, and then calculate its Hopf-dihedral
cohomology.  The next section is devoted to the dual Hopf-dihedral
homology together with its own characteristic map. In
Section~\ref{Sect6}, we investigate the interaction of our
characteristic map on the Hopf-dihedral cohomology, its dual, and the
Chern character from $L$-theory to dihedral homology on two examples:
(i) the group ring $k\pi$ of the fundamental group $\pi$ of a multiply
connected manifold $M$, and (ii) Podleś sphere $\C{O}(S_{qs}^2)$.

\subsection*{Notation and conventions}

We use a base field $k$ of characteristic 0 which carries a complex
structure.  WLOG one can assume $k=\B{Q}[\sqrt{-1}]$, or any other
field containing $\B{Q}[\sqrt{-1}]$.  We will use
$\B{N}\left<\frac{1}{2}\right>$ to denote the set of positive half
integers $0,\frac{1}{2},1,\frac{3}{2},\ldots$, and $\left<X\right>$ to
denote the $k$-vector space spanned by a given set $X$.  We use
$\CH_\bullet$, $\CC_\bullet$ and $\CDD_\bullet$ to denote respectively
the Hochschild, the cyclic and dihedral complexes associated with a
dihedral module.  Similarly, we use $HH_\bullet$, $HC_\bullet$ and
$HD_\bullet$ to denote respectively the Hochschild, the cyclic and the
dihedral homology of a dihedral module.  We use $\tor^A_\bullet$ and
$\ext_A^\bullet$ to denote respectively the derived functors of the
tensor product $\otimes_A$ and the Hom-functor $\Hom_A$ for a unital
associative algebra $A$.

\section{Cyclic and dihedral (co)modules and their
  (co)homology}\label{Sect1}

\subsection{The cyclic category}

Our main references for this subsection are \cite{Conn83} and
\cite{Loday-book}.

The cyclic category $\D C$ is the category with the set of objects $[n]$, $n\in \Nb$. The
morphisms, on the other hand, are generated by the cofaces $\partial_i: [n-1]\lra [n]$, 
$0\leq i \leq n$, the codegeneracies $\s_j:[n+1]\lra [n]$, $0\leq j\leq n$, and the cyclic operators $\tau_n:[n]\lra [n]$ subject to the relations
\begin{equation*}
\partial_j  \partial_i = \partial_i  \partial_{j-1}, \, \, i < j  , \qquad \s_j \s_i = \s_i\s_{j+1},  \, \,  i \leq j
\end{equation*}
\begin{equation*}
\s_j  \partial_i = \left\{ \begin{matrix} \partial_i  \s_{j-1} \hfill &i < j
\hfill \cr \Id_{[n]} \hfill &\hbox{if} \ i=j \ \hbox{or} \ i = j+1 \cr
\partial_{i-1}  \s_j \hfill &i > j+1 \hfill \cr
\end{matrix} \right.
\end{equation*}
\begin{eqnarray*}
\tau_n  \partial_i  = \partial_{i-1}  \tau_{n-1} ,
 \quad && 1 \leq i \leq n ,  \quad \tau_n  \partial_0 =
\partial_n \\ \label{cj} \tau_n  \s_i = \s_{i-1} \tau_{n+1} , \quad &&
1 \leq i \leq n , \quad \tau_n  \s_0 = \s_n  \tau_{n+1}^2 \\
\label{ce} \tau_n^{n+1} &=& \Id_{[n]}  \, .
\end{eqnarray*}


We call a functor of the form $F: \D C^{\rm op} \lra \lmod{k}$ a
cyclic module, and a functor of the form $G: \D C\lra \lmod{k}$ a
cocyclic module.

\subsection{The dihedral category}

Our main references for this subsection are \cite{Loda87} and
\cite{Loday-book}.

There is a straight extension of $\D C$ to a larger category $\D D$
called the dihedral category, with the same set of objects, containing
$\D C$ as a subcategory.  The essential difference is that the
endomorphisms on each object $[n]$ form the dihedral group
$D_{n+1} = \left<\tau_n,\omega_n\mid \omega_n^2=\tau_n^{n+1}=\omega_n
  \tau_n\omega_n\tau_n=\Id_{[n]}\right>$
of order $2(n+1)$.  The rest of the relations, between the morphisms
are
\begin{equation*} 
\partial_i \omega_n  = \omega_{n-1}\partial_{n-i},\qquad \s_i\omega_n =  \omega_{n+1} \s_{n-i},\qquad 0\leq i\leq n.
\end{equation*}
Similarly as before, a functor $F: \D D^{\rm op} \lra \lmod{k}$ is
called a dihedral module.  In the opposite case, we call a functor
$F: \D D\lra \lmod{k}$ a codihedral module.

\subsection{Cohomology} Our main references for this section are
\cite{Conn83,Loday-book}.

We define a cosimplicial module $\CH_\bullet^\D$
\begin{equation*}
  \CH_n^\D =  \left<\D(n,\ \cdot\ )\right>
\end{equation*}
and the face maps $\partial_i\colon\CH_n\to \CH_{n-1}$ are defined by
pre-composition
\begin{equation*}
  \partial_i(\psi) = \psi\circ \partial_i\in \D(n-1,m)
\end{equation*}
for all $\psi\colon n\to m$. Then we define a differential
$d_n = \sum_{i=0}^n (-1)^i\partial_i$.  This differential graded
cosimplicial module is a resolution of the cosimplicial module
$k_\bullet$.  If we let
$\CH^{\D C}_n = \CB^{\D C}_n = \left<\D C(n,\ \cdot\ )\right>$ and we
let
\begin{equation*}
  d^\CH_n = \sum_{i=0}^n (-1)^i\partial_i
  \quad \text{ and }\quad
  d^\CB_n = \sum_{i=0}^{n-1} (-1)^i\partial_i
\end{equation*}
with $N_\bullet = \sum_{i=0}^n t_n^i$, then the bicomplex $\CC_\bullet$
\begin{equation*}
  \CH_\bullet\xla{1-t_\bullet} \CB_\bullet \xla{N_*} \CH_\bullet\xla{1-t_\bullet}\cdots
\end{equation*}
is a resolution of the cocyclic module $k_\bullet$ in the category of
cocyclic modules.  Then we also would see that 
\begin{equation*}
  \CC_\bullet\xla{1-\omega_\bullet} \CC_\bullet \xla{1+\omega_\bullet} \CC_\bullet
  \xla{1-\omega_\bullet}\cdots
\end{equation*}
is going to be a projective resolution of $k_\bullet$ in the category
of codihedral modules, once we replace $\D C$ with $\D D$ in the
definitions of $\CB_\bullet$ and $\CH_\bullet$.

\begin{definition}
Let $\C{Z}$ be one of $\D$, $\D C$, or $\D D$, $X_\bullet$ a right $\C{Z}$-module, and $Y_\bullet$ a left
  $\C{Z}$-module.  Then the (co)homology of $X_\bullet$ and
  $Y_\bullet$ are defined as
  \begin{equation*}
    H\C{Z}_\bullet(X_\bullet) = \tor^\C{Z}_\bullet(X_\bullet,k_\bullet)
    \quad\text{ and }\quad
    H\C{Z}^\bullet(Y_\bullet) = \ext_\C{Z}^\bullet(k_\bullet,Y_\bullet)
  \end{equation*}
\end{definition}

\section{Algebras and coalgebras}\label{Sect2}

\subsection{Cyclic (co)homology of algebras}

With these notations at hand we can interpret the Hochschild and the cyclic (co)homology as derived functors. Namely, given an algebra $\mathcal{A}$ we define a cyclic module $\mathbf{C}_\bullet(\C{A})\to \lmod{k}$ by 

\begin{equation*}
\mathbf{C}_\bullet(\C{A})= \bigoplus_{n\geq 0} \C{A}^{\ot\,n+1}.
\end{equation*}
whose structure maps are defined as follows:
\begin{align*}
  \partial_i(a_0\otimes\cdots\otimes a_n) = & 
  \begin{cases}
    (\cdots\otimes a_{-1}\otimes a_i a_{i+1}\otimes a_{i+2}\otimes\cdots) & \text{ if } 0\leq i\leq n-1\\
    (a_na_0\otimes a_1\otimes\cdots\otimes a_{n-1}) & \text{ if } i=n
  \end{cases}\\
  \s_j(a_0\otimes\cdots\otimes a_n) = & (\cdots\otimes a_i\otimes 1\otimes a_{i+1}\otimes\cdots)\\
  \tau_n(a_0\otimes\cdots\otimes a_n) = & a_n\otimes a_0\otimes\cdots\otimes a_{n-1}
\end{align*}
for $0\leq i\leq n+1$ and $0\leq j\leq n$.  

One can realize the Hochschild and the cyclic homologies as
derived functors of simplicial and cyclic modules
\begin{equation*} 
HH_n(\C{A}) = \tor^{\D}_n(\mathbf{C}_\bullet(\C{A}),k_\bullet) \quad \text{ and } \quad HC_n(\C{A}) = \tor^{\D C}_n(\mathbf{C}_\bullet(\C{A}),k_\bullet).
\end{equation*}
For the Hochschild and cyclic cohomologies we get
\begin{equation*} 
HH^n(\C{A}) = \ext_{\D}^n(k_\bullet,\mathbf{C}^\bullet(\C{A})) \quad \text{ and } \quad HC^n(\C{A}) = \ext_{\D C}^n(k_\bullet,\mathbf{C}^\bullet(\C{A})),
\end{equation*}
where, this time, our cocyclic module $\mathbf{C}^\bullet(\C{A})$ is given by $\mathbf{C}^n(\C{A}) = \Hom(\C{A}^{\ot\,n+1},k)$.  More generally, for an algebra $\C{A}$ and an $\C{A}$-bimodule $V$, the cocyclic module $\mathbf{C}^n(\C{A},V) := \Hom(\C{A}^{\ot\,n},V)$ has the coface maps $d_i\colon \mathbf{C}^{n-1}(\C{A},V) \lra \mathbf{C}^n(\C{A},V)$
\begin{align*}
  d_i\vp(a_1,\ldots, a_n) = &
  \begin{cases}
    a_1\cdot \vp(a_2\ldots, a_n),            & \text{ if } i = 0\\
    \vp(a_1,\ldots, a_i a_{i+1},\ldots, a_n), & \text{ if } 1\leq i \leq n-1 \\
    \vp(a_1,\ldots, a_{n-1})\cdot a_n,        & \text{ if } i=n
  \end{cases}
\end{align*}
and the codegeneracy maps
$s_j\colon\mathbf{C}^{n+1}(\C{A},V) \lra \mathbf{C}^n(\C{A},V)$
\begin{align*}
 s_j\vp(a_1,\ldots, a_n) = \vp(a_1,\ldots, a_j,1,a_{j+1},\ldots, a_n).
\end{align*}
where $0\leq j\leq n$.  In particular, for $V=\C{A}^\vee = \Hom(\C{A},k)$, the above
structure is equivalent to the one given by
$d_i\colon \mathbf{C}^{n-1}(\C{A}) \lra \mathbf{C}^n(\C{A})$
\begin{align*}
  d_i\vp(a_0,\ldots, a_n) = & 
   \begin{cases}
     \vp(a_0,\ldots, a_i a_{i+1},\ldots, a_n), & \text{ if } 0\leq i\leq n-1 \\
     \vp(a_n a_0,\ldots, a_{n-1}),             & \text{ if } i=n
   \end{cases}
\end{align*}
and $s_j\colon\mathbf{C}^{n+1}(\C{A}) \lra \mathbf{C}^n(\C{A})$ 
\begin{align*}
s_j\vp(a_0,\ldots, a_n) = \vp(a_0,\ldots, a_j,1,a_{j+1},\ldots, a_n).
\end{align*}
where $0\leq j\leq n$.  In this case, the cyclic maps
$t_n:\mathbf{C}^n(\C{A}) \lra \mathbf{C}^n(\C{A})$ are given by
\begin{align*}
 t_n\vp(a_0,\ldots, a_n) = \vp(a_n, a_0,\ldots, a_{n-1}).
\end{align*}
for every $n\geq 1$.

\subsection{Hochschild (co)homology of $\ast$-algebras}

Let $\C{A}$ be an algebra with an involution $*:\C{A}\lra \C{A}$, that is,
\begin{equation*}
(ab)^*=b^*a^*,\quad a^{**}=a
\end{equation*}
for any $a,b\in \C{A}$. We call such an algebra a $*$-algebra. Let also $V$ be an $\C{A}$-bimodule with an involution $*:V\lra V$ satisfying
\begin{equation*}
(a\cdot v\cdot b)^* = b^* \cdot v^* \cdot a^*,\quad v^{**}=v
\end{equation*}
for any $a,b\in \C{A}$, and any $v\in V$. Such a bimodule is called a $*$-bimodule. One then defines on the Hochschild cohomology complex
\begin{equation*}
w_n\vp (a_1,\ldots, a_n) := \vp(a_n^*,a_{n-1}^*,\ldots,a_1^*)^*
\end{equation*}
for any $\vp \in \mathbf{C}^n(\C{A},V)$.

\begin{lemma}
On $\mathbf{C}^{n-1}(\C{A},V)$ we have $w_n d_i = d_{n-i}w_{n-1}$ for $0\leq i \leq n$.
\end{lemma}

\begin{proof}
For $1\leq i\leq n-1$, and for $\vp\in \mathbf{C}^{n-1}(\C{A},V)$, we have
\begin{align*}
w_n(d_i\vp)(a_1,\ldots, a_n) 
 = & (d_i\vp)(a_n^*,\ldots, a_1^*)^* \\
 = & \vp(a_n^*,\ldots, a_{n-i+1}^*a_{n-i}^*,\ldots, a_1^*)^* \\
 = & \vp(a_n^*,\ldots, (a_{n-i}a_{n-i+1})^*,\ldots, a_1^*)^* \\
 = & d_{n-i}(w_{n-1}\vp)(a_1,\ldots, a_n).
\end{align*}
Similarly, $w_n d_0 = d_{n}w_{n-1}$, and $w_n d_n = d_0w_{n-1}$.
\end{proof}

As a result, we have the following:

\begin{lemma}
Given a $*$-algebra $\C{A}$ and a $*$-bimodule $V$, on $\mathbf{C}^{n-1}(\C{A},V)$ we have
\begin{equation*}
w_n b = (-1)^n b w_{n-1}.
\end{equation*}
\end{lemma}

\begin{proof}
For an arbitrary $\vp \in \mathbf{C}^{n-1}(\C{A},V)$ we have
\begin{align*}
w_n b\vp(a_1,\ldots, a_n)
 = & b\vp(a_n^*,\ldots, a_1^*)^* \\
 = & a_n^*\cdot \vp(a_{n-1}^*, \ldots, a_1^*)^* \\
   & + \sum_{i= 1}^{n-1}(-1)^i\vp(a_n^*,\ldots, a_{n-i+1}^*a_{n-i}^*,\ldots, a_1^*)^* \\
   & + (-1)^n\vp(a_n^*,\ldots, a_2^*)\cdot a_1^* \\
 = & (-1)^nbw_{n-1}\vp(a_1,\ldots, a_n).
\end{align*}
\end{proof}

Now, assuming $2$ is invertible in $k$, based on the decomposition of the Hochschild complex into the $+1$ and $-1$-eigenspaces of the
involution operator it follows that
\begin{equation*}
HH^\bullet(\C{A},V) = HH^\bullet_+(\C{A},V) \oplus HH^\bullet_-(\C{A},V),
\end{equation*}
see for instance \cite[5.2.3]{Loday-book}, and in particular for $V=\C{A}^\vee=\Hom(\C{A},k)$,
\begin{equation*}
  HH^\bullet(\C{A}) = HH^\bullet_+(\C{A}) \oplus HH^\bullet_-(\C{A}).
\end{equation*}
Explicitly, $HH^\bullet_+(\C{A},V)$ (resp. $HH^\bullet_+(\C{A})$) is the cohomology of the $w$-invariant (real) subcomplex, and $HH^\bullet_-(\C{A},V)$ (resp. $HH^\bullet_-(\C{A})$) is the cohomology of the $w$-anti-invariant (imaginary) subcomplex.

\subsection{Dihedral cohomology of $*$-algebras}

We next record the following on the restriction of the $*$-structure on the cyclic complex.
\begin{lemma}
On $\mathbf{C}^n(\C{A}):=\mathbf{C}^n(\C{A},\C{A}^\ast)$, we have 
\begin{equation*}
t_nw_n = w_nt_n^{-1}.
\end{equation*}
\end{lemma}

\begin{proof}
For any $\vp\in \mathbf{C}^n(\C{A})$, we first recall that $w_n(\vp)(a_0,\ldots, a_n)=\wbar{\vp(a_0^*,a_n^*,\ldots,a_1^*)}$. Then the claim follows from
\begin{align*}
  t_nw_n\vp(a_0,\ldots, a_n)
  = & w_n\vp(a_n,a_0,\ldots, a_{n-1})\\
  = & \wbar{\vp(a_n^*,a_{n-1}^*,\ldots, a_0^*)} \\
  = & \wbar{t_n^{-1}\vp(a_0^*,a_n^*,\ldots, a_1^*)}\\
  = & w_nt_n^{-1}\vp(a_0,\ldots, a_n).
\end{align*}
\end{proof}

As a result, if $\vp \in \mathbf{C}_\lambda^n(\C{A})$, \ie $t_n\vp = (-1)^n\vp$, then $t_nw_n\vp = w_nt_n^{-1}\vp = (-1)^nw_n\vp$, that is $w_n\vp \in \mathbf{C}_\lambda^n(\C{A})$, therefore, we obtain the similar eigen-space decomposition
\begin{equation*} 
HC^\bullet(\C{A}) = HC^\bullet_+(\C{A})\oplus HC^\bullet_-(\C{A}).
\end{equation*}
The summands are both called the dihedral cohomology of $\C{A}$, and are denoted by $HD^\bullet_\pm(\C{A})$.

\subsection{Dihedral cohomology of $*$-coalgebras}

Let us recall from \cite[Sect. 2.2]{FariSolo96} the involutive coalgebras and their (involutive) comodules. A $*$-coalgebra $\C{C}$ is a coalgebra such that
\begin{equation*} 
\D(c^*) = c_{(2)}^*\otimes c_{(1)}^*, \qquad \ve(c^*) = \wbar{\ve(c)}.  
\end{equation*}
An involutive $\C{C}$-bicomodule ($*$-bicomodule) $V$ is a $\C{C}$-bicomodule such that 
\begin{equation}\label{star-bicomodule}
(v^*)\ns{0}\ot(v^*)\ns{1} = {v\ns{0}}^*\ot {v\ns{-1}}^*,\qquad (v^*)\ns{-1}\ot(v^*)\ns{0}={v\ns{1}}^* \ot {v\ns{0}}^*.
\end{equation}
We recall from \cite{Doi81} the coalgebra Hochschild cohomology of a coalgebra $\C{C}$ with coefficients in a $\C{C}$-bicomodule $V$ is given as the homology of the complex
\begin{equation*}
\mathbf{C}^\bullet(\C{C},V) = \bigoplus_{n\geq 0} V\ot \C{C}^{\ot \,n}
\end{equation*}
with the structure maps $d_i\colon \mathbf{C}^{n-1}(\C{C},V) \lra \mathbf{C}^n(\C{C},V)$
\begin{align*}
d_i\vp(v\ot c_1 \odots c_{n-1}) = &
  \begin{cases}
    v\ns{0}\ot v\ns{1}\ot c_1 \odots c_{n-1}, & \text{ if } i=0,\\
    v\ot c_1 \odots \D(c_i) \odots c_{n-1},   & \text{ if } 1\leq i \leq n-1, \\
    v\ns{0}\ot c_1 \odots c_{n-1}\ot v\ns{-1} & \text{ if } i=n,
  \end{cases}
\end{align*}
and $s_j\colon\mathbf{C}^{n+1}(\C{C},V) \lra \mathbf{C}^n(\C{C},V)$
\begin{align*}
 s_j(v\ot c_0 \odots c_n) = v\ot c_0 \odots \ve(c_j)\odots c_n,
\end{align*}
where $0\leq j\leq n$.

\begin{lemma}\label{HopfHochschildCompatibleWithStar}
On $\mathbf{C}^{n-1}(\C{C},V)$ we have $w_n d_i = d_{n-i}w_{n-1}$ for $0\leq i \leq n$.
\end{lemma}

\begin{proof}
For $1\leq i\leq n-1$, and for $v\ot c_1 \odots c_{n-1}\in \mathbf{C}^{n-1}(\C{C},V)$, we have
\begin{align*}
  w_n d_i(v\ot c_1 \odots c_{n-1})
  = & w_n(v\ot c_1 \odots {c_i}\ps{1}\ot {c_i}\ps{2} \odots c_{n-1}) \\
  = & v^*\ot c_{n-1}^* \odots ({c_i}\ps{2})^*\ot ({c_i}\ps{1})^* \odots c_1^*\\
  = & d_{n-i}(v^*\ot c_{n-1}^* \odots c_i^* \odots c_1^*) \\
  = & d_{n-i}w_{n-1}(v\ot c_1 \odots c_{n-1}).
\end{align*}
Similarly, $w_n d_0 = d_{n}w_{n-1}$, and $w_n d_n = d_0w_{n-1}$.
\end{proof}

We thus conclude the commutation with the Hochschild coboundary map.

\begin{corollary}
Given a $*$-coalgebra $\C{C}$ and a $*$-bicomodule $V$, on $\mathbf{C}^{n-1}(\C{C},V)$ we have
\begin{equation*}
w_n b = (-1)^n b w_{n-1}.
\end{equation*}
\end{corollary}

\begin{proof}
For an arbitrary $v\ot c_1 \odots c_{n-1} \in \mathbf{C}^{n-1}(\C{C},V)$ we have
\begin{align*}
w_n b(v\ot c_1 & \odots c_{n-1}) \\
 = & ({v\ns{0}}^* \ot c_{n-1}^* \odots c_1^* \ot {v\ns{1}}^*) \\
   & + \sum_{i= 1}^{n-1}(-1)^i (v \ot c_{n-1}^*\odots ({c_i}\ps{2})^*\ot ({c_i}\ps{1})^*\odots c_1^*)\\
   & + (-1)^{n} ({v\ns{0}}^*\ot {v\ns{-1}}^* \odots c_{n-1}^*\ot c_1^*)\\
 = & (-1)^n b w_{n-1}(v\ot c_1 \odots c_{n-1}).
\end{align*}
\end{proof}

As a result, we have the decomposition
\begin{equation*}
HH^\bullet(\C{C},V) = HH^\bullet_+(\C{C},V) \oplus HH^\bullet_-(\C{C},V),
\end{equation*}
of the coalgebra Hochschild cohomology into the dihedral coalgebra Hochschild cohomologies.

\section{Hopf-dihedral cohomology}\label{Sect3}

\subsection{Hopf-cyclic cohomology}

Let us first recall Hopf-cyclic cohomology from \cite{ConnMosc98,ConnMosc}. Let $\C{H}$ be a Hopf algebra with a modular pair in involution (MPI) $(\d,\s)$, \ie $\d$ is a character on $\C{H}$, and $\s\in \C{H}$ is a group-like element such that
\begin{equation*}
S_\d^2 = \Ad_\s, \qquad \d(\s)=1,
\end{equation*}
where $S_\d(h) = \d(h_{(1)})S(h_{(2)})$ is the twisted antipode. Assume also that $\C{A}$ is a (left) $\C{H}$-module algebra,
\begin{equation*}
h\rt (ab) = (h_{(1)}\rt a)(h_{(2)}\rt b), \qquad h\rt 1=\ve(h)1,
\end{equation*}
equipped with a linear form $\tau:\C{A}\lra k$ which is a $\d$-invariant $\s$-trace
\begin{equation*}
\tau(h\rt a) = \d(h)\tau(a), \qquad \tau(ab) = \tau(b(\s\rt a)).
\end{equation*}
Then the Hopf-cyclic cohomology of $\C{H}$ is defined to satisfy the following, \cite{ConnMosc}.

{\bf Ansatz:} Let $\C{A}$ be a $\C{H}$-module algebra equipped with a $\d$-invariant $\s$-trace $\tau:\C{A}\lra k$. Then the assignment
\begin{align*}
& h_1\odots h_n \mapsto \chi_\tau(h_1\odots h_n) \in \mathbf{C}^n(\C{A}), \\
& \chi_\tau(h_1\odots h_n)(a_0,\ldots, a_n)=\tau(a_0(h_1\rt a_1)\ldots (h_n\rt a_n))
\end{align*}
defines a cocyclic module $\mathbf{C}^\bullet(\C{H};\s,\d)$ on $\C{H}$ whose cohomology comes with a canonical map of the form $\chi_\tau\colon HC^n(\C{H};\d,\s)\lra HC^n(\C{A})$.

The ansatz then dictates on 
\begin{equation*}
\mathbf{C}^\bullet(\C{H};\s,\d) = \bigoplus_{n\geq 0} \C{H}^{\otimes n}
\end{equation*}
the cocyclic structure given by the maps
\begin{equation*}
d_i(h_1\odots h_{n-1}) = 
   \begin{cases}
     1\otimes h_1\odots h_{n-1} & \text{ if } i=0,\\
     h_1\odots \D(h_i)\odots h_{n-1} & \text{ if } 0<i\leq n,\\
     h_1 \odots h_n \otimes \sigma & \text{ if } i = n,
   \end{cases} 
\end{equation*}
and
\begin{equation*}
s_j(h_0\odots h_n) = h_1 \odots \ve(h_j)\odots h_n,\qquad 0\leq j\leq n.
\end{equation*}
These structure maps encode a simplicial structure $\mathbf{C}^\bullet(\C{H},\s,\d)\colon\D\to \lmod{k}$, and the maps
\begin{equation*}
 t_n(h_1 \odots h_n) = S_\d(h_1)\cdot (h_2 \odots h_n \ot \sigma)
\end{equation*}
encode the actions of the cyclic groups to the cocyclic module $\mathbf{C}^\bullet(\C{H};\s,\d)$ which is a functor of the form $\mathbf{C}^\bullet(\C{H},\s,\d)\colon\D C\to \lmod{k}$.

\subsection{Hopf $*$-algebras}

Let us next recall from \cite[Def. 1.7.5]{Majid-book} that a Hopf $*$-algebra is a Hopf algebra $\C{H}$, which is a $*$-algebra such that
\begin{align*}
\D(h^*) = {h\ps{1}}^*\ot {h\ps{2}}^*, \quad \ve(h^*)=\wbar{\ve(h)},\qquad S(h^*) = S^{-1}(h)^*
\end{align*}
for every $h\in\C{H}$.

Let now a $*$-algebra $\C{A}$ be a $\C{H}$-module algebra with the $*$-action, that is, let $\C{A}$ be a $\C{H}$-module algebra with the compatibility
\begin{equation}\label{star-S}
(h\rt a)^* = S^{-1}(h^*)\rt a^*
\end{equation}
between the $*$-structures and the action. Such $\C{A}$ is called a $\C{H}$-module $*$-algebra. See \cite[Prop. 6.1.5]{Majid-book} for the notion of $*$-action, see also \cite[Eqn. (8)\&(9)]{SchmWagn04}.

\subsection{Hopf-dihedral cohomology of a Hopf $*$-algebra}

We now investigate the $*$-structure on the Hopf-cyclic complex $\mathbf{C}(\C{H};\s,\d)$ that makes the characteristic homomorphism of Connes-Moscovici a 
$*$-homomorphism, that is, for $n \geq 0$
\begin{equation*}
w_n(\chi_\tau(h^1\odots h^n)) = \chi_\tau(w_n(h^1\odots h^n)).
\end{equation*} 
Accordingly, 
\begin{align*}
w_n(\chi_\tau(h^1\odots h^n))(a_0,\ldots, a_n) 
  = & \wbar{\chi_\tau(h^1\odots h^n)(a_0^*, a_n^*,\ldots, a_1^*)} \\
  = & \wbar{\tau(a_0^*(h^1\rt a_n^*)\ldots (h^n\rt a_1^*))} \\
  = & \wbar{\tau(a_0^*(S(h^1)^*\rt a_n)^*\ldots (S(h^n)^*\rt a_1)^*)} \\
  = & \tau((S(h^n)^*\rt a_1) \ldots (S(h^1)^*\rt a_n)a_0)\\
  = & \tau((S^{-1}({h^n}^*)\rt a_1) \ldots (S^{-1}({h^1}^*)\rt a_n)a_0) \\
  = & \tau(a_0(\s S^{-1}({h^n}^*)\rt a_1) \ldots (\s S^{-1}({h^1}^*)\rt a_n))
\end{align*}
dictates that
\begin{equation}\label{star-Hopf-cyclic}
w_n(h^1\odots h^n) = \s S^{-1}({h^n}^*) \odots \s S^{-1}({h^1}^*).
\end{equation}

\begin{lemma}
The mapping given by \eqref{star-Hopf-cyclic} is an involution on $\mathbf{C}^\bullet(\C{H};\s,\d)$.
\end{lemma}

\begin{proof}
For $n\geq 1$ we have
\begin{align*}
  w_n^2(h^1\odots h^n)
  = & w_n(\s S^{-1}({h^n}^*) \odots \s S^{-1}({h^1}^*)) \\
  = & \s S^{-1}(S^{-1}({h^1}^*)^* \s^*) \odots \s S^{-1}(S^{-1}({h^n}^*)^*\s^*) \\
  = & S^{-1}(S^{-1}({h^1}^*)^*) \odots S^{-1}(S^{-1}({h^n}^*)^*) \\
  = & S^{-1}(S({h^1}^{**})) \odots S^{-1}(S({h^n}^{**}))\\
  = & h^1\odots h^n.
\end{align*}
On the third equation we used the assumption that $\s^*=\s$, on the second, third and the fourth equations we used \eqref{star-S}.
\end{proof}

\begin{lemma}
The mapping given by \eqref{star-Hopf-cyclic} is an involution on the coalgebra Hochschild cohomology $HH^\bullet(\C{H},{}^\sigma k)$.
\end{lemma}

\begin{proof}
For $1\leq i \leq n-1$ we observe on $\mathbf{C}^{n-1}(\C{H};\s,\d)$ that
\begin{align*}
 w_n d_i(h^1\odots h^{n-1})
  = & w_n(h^1\odots {h^i}\ps{1}\ot {h^i}\ps{2} \odots h^{n-1}) \\
  = & \s S^{-1}({h^{n-1}}^*) \odots \s S^{-1}({{h^i}\ps{2}}^*) \ot \s S^{-1}({{h^i}\ps{1}}^*) \odots \s S^{-1}({h^1}^*) \\
  = & \s S^{-1}({h^{n-1}}^*) \odots \D(\s S^{-1}({h^i}^*)) \odots \s S^{-1}({h^1}^*) \\
  = & d_{n-i}w_{n-1}(h^1\odots h^{n-1}).
\end{align*}
The equalities $w_n d_0 = d_nw_{n-1}$ and $w_n d_n = d_0w_{n-1}$ follows similarly from \eqref{star-bicomodule}.
\end{proof}

\begin{lemma}
The mapping given by \eqref{star-Hopf-cyclic} is an involution on the Hopf-cyclic cohomology $HC^\bullet(\C{H};\s,\d)$.
\end{lemma}

\begin{proof}
Let $\wtilde{h}:=h_1\odots h_n \in \mathbf{C}^n(\C{H};\s,\d)$ be cyclic. Then,
\begin{align*}
  \chi_\tau(t_nw_n(\wtilde{h})) 
  = & t_n\chi_\tau(w_n(\wtilde{h})) = t_nw_n\chi_\tau(\wtilde{h}) = w_nt_n^{-1}\chi_\tau(\wtilde{h}) = \chi_\tau(w_nt_n^{-1}(\wtilde{h}))\\
  = & (-1)^n\chi_\tau(w_n(\wtilde{h})),
\end{align*}
that is $\om_n(\wtilde{h}) \in \mathbf{C}^n(\C{H};\s,\d)$ is also cyclic.
\end{proof}

Consequently, 
\begin{equation*}
HC^\bullet(\C{H};\s,\d) = HC^\bullet_+(\C{H};\s,\d) \oplus HC^\bullet_-(\C{H};\s,\d).
\end{equation*}
We call the cohomologies on the right hand side, corresponding to the $\pm 1$-eigenspaces of the operator \eqref{star-Hopf-cyclic}, the Hopf-dihedral cohomologies of the Hopf algebra $\C{H}$.

\subsection{Hopf-dihedral cohomology with general coefficients}

Let $\C{H}$ be a Hopf algebra with an invertible antipode, and
$\C{A}$ be a left $\C{H}$-module algebra. 
Let also $V$ be a left $\C{H}$-module and left $\C{H}$-comodule
satisfying the stability condition $v\ns{-1}\tr v\ns{0} = v$, for all
$v\in V$.  Then one can define a para-cyclic module
$\CC_\bullet^\C{H}(\C{A},V)$ letting
\begin{align*}
  \CC_n^\C{H}(\C{A},V) = \C{A}^{\otimes n+1}\otimes V
\end{align*}
and defining the structure maps
\begin{align*}
  \partial_i(a_0\otimes\cdots\otimes a_n\otimes v) = &
  \begin{cases}
    (\cdots \otimes a_{i-1}\otimes a_ia_{i+1}\otimes a_{i+2}\otimes\cdots\otimes v)
         & \text{ if } 0\leq i\leq n-1,\\
    (v\ns{-1}\tr a_n)a_0\odots a_{n-1}\otimes v\ns{0}
         & \text{ if } i=n,
  \end{cases}\\
  \sigma_j(a_0\otimes\cdots\otimes a_n\otimes v)
     = & a_0 \odots a_j\otimes 1\otimes a_{j+1}\odots a_n\ot  v, \quad \text{ for } 0\leq j\leq n-1,\\
  \tau_n(a_0\otimes\cdots\otimes a_n\otimes v)
     = & (v\ns{-1}\tr a_n)\otimes a_0\odots a_{n-1}\otimes v\ns{0}.
\end{align*}
If we let
$\CC^\bullet_\C{H}(\C{A},V) :=
\Hom_\C{H}(\CC_\bullet^\C{H}(\C{A},V),k)$, then
we get a cocyclic module. Indeed, for every
$\vp\in \CC^n_\C{H}(\C{A},V)$ we have
\begin{align*}
    (\tau_n^{n+1}\vp)(a_0\otimes\cdots\otimes a_n\otimes v)
    = & \vp(\tau_n^{n+1}(a_0\otimes\cdots\otimes a_n\otimes v))\\
    = & \vp(v\ns{-n-1}\tr a_0\otimes\cdots\otimes v\ns{-1}\tr a_n\otimes v\ns{0})\\
    = & \vp(v\ns{-n-2}\tr a_0\otimes\cdots\otimes v\ns{-2}\tr a_n\otimes v\ns{-1}\tr v\ns{0})\\
    = & \varepsilon(v\ns{-1})\vp(a_0\otimes\cdots\otimes a_n\otimes v\ns{0})\\
    = & \vp(a_0\otimes\cdots\otimes a_n\otimes v)
\end{align*}
We call the cyclic cohomology of this cocyclic module
$\CC^\bullet_\C{H}(\C{A},V)$ as the Hopf-cyclic cohomology of the
$\C{H}$-module algebra $\C{A}$ with coefficients in $V$, and we
denote it by $HC^\bullet_\C{H}(\C{A},V)$.

\begin{proposition}\label{CohomologicalPairing}
  Given any stable $\C{H}$-module/comodule $V$, and
  $\tau\in HC^0_\C{H}(\C{A},V)$ we have a pairing of the form
  $\left<\cdot|\cdot\right>_\tau\colon \C{H}^{\otimes n}\otimes V\otimes\C{A}^{\otimes n+1}\to
  k$ given by
  \begin{align}\label{eq:CohomologicalPairing}
    \left<h^1\otimes\cdots\otimes h^n\otimes v \mid a_0\otimes\cdots\otimes a_n\right>_\tau
    = \tau(a_0(h^1\tr a_1)\cdots(h^n\tr a_n)\otimes v).
  \end{align}
The pairing induces a cocyclic module structure on $\CC^\bullet_\C{H}(\C{H},V)$, whose
  cohomology $HC^\bullet_\C{H}(\C{H},V)$ comes with a characteristic map
  $\chi_\tau\colon HC^\bullet_\C{H}(\C{H},V)\to HC^\bullet(\C{A})$.  In case $V$
  is a left/left SAYD module over $\C{H}$, this cohomology is the Hopf-cyclic cohomology with
  coefficients in $V$.
\end{proposition}

\begin{proof}
  We need to derive the structure maps on
  $\CC^\bullet_\C{H}(\C{H},V)$ compatible with the pairing. For that we first
  observe
  \begin{align*}
    \left<h^1\otimes\cdots\otimes\right. & \left.h^n\otimes v \mid \partial_i(a_0\otimes\cdots\otimes a_{n+1})\right>_\tau\\ = & 
       \begin{cases}
          \tau(a_0a_1(h^1\tr a_2)\cdots (h^n\tr a_{n+1})\otimes v) & \text{ if } i=0,\\
          \tau(a_0\cdots(h^{i-1}\tr a_{i-1})(h^i\tr a_ia_{i+1})(h^{i+1}\tr a_{i+2})\cdots\otimes v)
                                                                 & \text{ if } 1\leq i\leq n,\\
          \tau(a_{n+1}a_0(h^1\tr a_1)\cdots(h^n\tr a_n)\otimes v)  & \text{ if } i=n+1,
       \end{cases}\\
     = & 
       \begin{cases}
          \tau(a_0a_1(h^1\tr a_2)\cdots (h^n\tr a_{n+1})\otimes v) & \text{ if } i=0,\\
          \tau(a_0\cdots(h^{i-1}\tr a_{i-1})({h^i}\ps{1}\tr a_i)({h^i}\ps{2}\tr a_{i+1})(h^{i+1}\tr a_{i+2})\cdots\otimes v)
                                                                 & \text{ if } 1\leq i\leq n,\\
          \tau(a_0(h^1\tr a_1)\cdots(h^n\tr a_n)(v\ns{-1}\tr a_{n+1})\otimes v\ns{0})  & \text{ if } i=n+1,
       \end{cases}
  \end{align*}
which forces the coface maps for the Hopf-cyclic cocyclic module with
  coefficients to be
  \begin{align*}
    d_i(h^1\otimes\cdots\otimes h^n\otimes v) = &
    \begin{cases}
       (1\otimes h^1\otimes\cdots\otimes h^n\otimes v) & \text{ if } i=0\\
       (\cdots\otimes h^{i-1}\otimes \Delta(h^i)\otimes h^{i+1}\otimes\cdots\otimes v)
                                                       & \text{ if } 1\leq i\leq n\\
       (h^1\otimes\cdots\otimes h^n\otimes v\ns{-1}\otimes v\ns{0})
                                                       & \text{ if } i=n+1.
    \end{cases}
  \end{align*}
  As for the codegeneracies we observe
  \begin{align*}
    \left<h^1\otimes\cdots\otimes\right. & \left.h^{n+1}\otimes v \mid s_j(a_0\otimes\cdots\otimes a_{n+1})\right>_\tau\\ = & 
        \begin{cases}
          \tau(a_0(h^1\tr 1)(h_2\tr a_1)\cdots(h^{n+1}\tr a_n)\otimes v) 
                      & \text{ if } j=0\\
          \tau(a_0(h^1\tr a_1)\cdots(h^j\tr a_j)(h^{j+1}\tr 1)(h^{j+2}\tr a_{j+1})\cdots(h^{n+1}\tr a_n)\otimes v)
                      & \text{ if } 1\leq j\leq n-1.
        \end{cases}
  \end{align*}
As a result, the codegeneracies need to be defined as
  \begin{align*}
    \s_j(h^1\otimes\cdots\otimes h^j)
    = (\cdots h^j\otimes \varepsilon(h^{j+1})\otimes h^{j+2}\otimes \cdots\otimes v)
  \end{align*}
  for $0\leq j\leq n-1$.  Finally, for the cyclic maps we get
  \begin{align*}
    \left<h^1\otimes\cdots\otimes\right. & \left.h^{n+1}\otimes v|\tau_n(a_0\otimes\cdots\otimes a_n)\right>_\tau\\ 
    = & \tau(a_n(h^1\tr a_0)\cdots\otimes (h^n\tr a_{n-1})\otimes v)\\
    = & \tau((h^1\tr a_0)\cdots\otimes (h^n\tr a_{n-1})(v\ns{-1}\tr a_n)\otimes v\ns{0})\\
    = & \tau(a_0(S({h^1}\ps{n+2})h_2\tr a_1)\cdots(S({h^1}\ps{3})h_n\tr a_{n-1})(S({h^1}\ps{2})v\ns{-1}\tr a_n)\otimes S({h^1}\ps{1})v\ns{0}),
  \end{align*}
  which means that the cyclic maps need to be defined as
  \begin{align*}
    \tau_n(h^1\otimes\cdots\otimes h^n\otimes v)
    = S({h^1}\ps{n+2})h^2\otimes\cdots\otimes S({h^1}\ps{3})h^n\otimes S({h^1}\ps{2})v\ns{-1}\otimes S({h^1}\ps{1})v\ns{0}.
  \end{align*}
  As for the agreement with the original Hopf-cyclic cohomology with
  coefficients, we observe that following \cite{Kaygun:CupProductII} the
  pairing is defined uniquely in cohomology since the terms
  \begin{align*}
    HC^p_\C{H}(\C{H},V)\otimes HC_\C{H}^q(\C{A},V)\to HC^{p+q}(\C{A})
  \end{align*}
  come from a derived bifunctor.  In other words, any cohomological
  pairing whose $0$-th term is given
  in~\eqref{eq:CohomologicalPairing} for $n=0$ will be the same with
  our pairing up to natural equivalence.
\end{proof}

\begin{theorem}\label{Hopf-cyclic-stable-coeff}
Let $\C{H}$ be a Hopf $*$-algebra, and $\C{A}$ an $\C{H}$-module
  $*$-algebra.  Let $V$ be a stable $\C{H}$-module/comodule together
  with a $*$-structure satisfying
  \begin{align*}
    (h\tr v)^* = S^{-1}(h^*)\tr v^*,\qquad 
    (v^*)\ns{-1}\otimes (v^*)\ns{0} = v^*\ns{-1}\otimes v^*\ns{0}
  \end{align*}
  for any $h\in\C{H}$ and any $v\in V$.  Assume also that there is a cocycle
  $\tau\in HC^0_\C{H}(\C{A},V)$ which additionally satisfies
  \begin{align*}
    \overline{\tau(a\otimes v)}
    = \tau(a^*\otimes v^*).
  \end{align*}
  Then the Hopf-cyclic cohomology $HC^\bullet_\C{H}(\C{H},V)$ carries a
  $*$-structure, and splits into two eigen-spaces
  \begin{align*}
    HC^\bullet_\C{H}(\C{H},V) = HC^\bullet_{\C{H},+}(\C{H},V) \oplus  HC^\bullet_{\C{H},-}(\C{H},V),
  \end{align*}
Furthermore, there are characteristic maps
  \begin{align*}
    \chi_\tau\colon HC^\bullet_{\C{H},\pm}(\C{H},V)\to HC^\bullet_\pm(\C{A}).
  \end{align*}
\end{theorem}

\begin{proof}
  In Proposition~\ref{CohomologicalPairing} we proved that our pairing
  is compatible with the cyclic structure.  What remains to be
  constructed is a *-structure which is compatible with our pairing.
  For that we observe
  \begin{align*}
    \left<h^1\otimes\cdots\otimes\right. & \left.h^{n+1}\otimes v|(a_0\otimes\cdots\otimes a_n)^*\right>_\tau\\ 
     = & \tau(a_0^*(h^1\tr a_n^*)\cdots(h^n\tr a_1^*)\otimes v)\\
     = & \overline{\tau((h^n\tr a_1^*)^*\cdots(h^1\tr a_n^*)^*a_0\otimes v^*)}\\
     = & \overline{\tau((S^{-1}(v\ns{-1}^*)\tr a_0)(S(h^n)^*\tr a_1)\cdots(S(h^1)^*\tr a_n)\otimes v\ns{0}^*)}\\
     = & \overline{\tau(a_0(v\ns{-n}^*S^{-1}({h^n}^*)\tr a_1)\cdots(v\ns{-1}^*S^{-1}({h^1}^*)\tr a_n)\otimes v\ns{0}^*)}.
  \end{align*}
  This means the $*$-structure needs to be defined as
  \begin{align*}
    (h^1\otimes\cdots\otimes h^n\otimes v)^*
    = v\ns{-n}^*S^{-1}({h^n}^*)\otimes\cdots\otimes v\ns{-1}^*S^{-1}({h^1}^*)\otimes v\ns{0}^*.
  \end{align*}
  As in the Hopf-cyclic case, the pairing
  \begin{align*}
    HC^p_{\C{H},\pm}(\C{H},V)\otimes HC^q_{\C{H},\pm}(\C{A},V)\to HC^{p+q}_\pm(\C{A})
  \end{align*}
  is defined uniquely in cohomology by~\cite{Kaygun:CupProductII}: any
  cohomological pairing whose $n=0$ term is given
  by~\eqref{eq:CohomologicalPairing} together with compatibility with
  the $*$-structures will be the same as ours up to natural
  equivalence.
\end{proof}

We call the cohomologies $HC^\bullet_{\C{H},\pm}(\C{H},V)$ as the
Hopf-dihedral cohomologies of $\C{H}$ with coefficients in a stable
$H$-module/comodule $V$.

\subsection{The differential graded $*$-algebra structure on the Hopf-Hochschild complex}

Let $\C{H}$ be a Hopf algebra, let $\C{G}(\C{H})$ denote the set of
group-like elements in $\C{H}$. 

Given any $\s\in \C{G}(\C{H})$, one can think
of $k$ both as a right $\C{H}$-comodule $k^\s$, and a left
$\C{H}$-comodule ${}^\s k$ by
\begin{equation}
  \rho_\s\colon k^\s\to k^\s \ot \C{H} \qquad \rho_\s(1) = 1\otimes \s
  \quad \text{ and }\quad
  \lambda_\s\colon {}^\alpha k \to {}^\alpha k \ot \C{H} \qquad \lambda_\s(1) = \s\otimes 1.
\end{equation}
It then follows that for any $\s\in\C{G}(\C{H})$ the Hochschild
complex $\CH^\bullet(\C{H},{}^\s k)$ is the same as the two sided cobar
complex $\CB^\bullet(k^1,\C{H},{}^\s k)$.

\begin{proposition}\label{star-product}
  Let $\C{H}$ be a Hopf algebra, such that the group-like elements
  $\C{G}(\C{H})$ forms an abelian group.  Then the Hochschild complex
  $\bigoplus_{\s\in\C{G}(\C{H})} \CH^\bullet(\C{H},{}^\s k)$ is a
  differential graded unital $*$-algebra with the product
  \begin{align*}
    \CH^p(\C{H},{}^{\s_1} k) & \otimes \CH^q(\C{H},{}^{\s_2} k) \to \CH^{p+q}(\C{H},{}^{\s_1\s_2} k)
  \end{align*}
  for any $p,q\in\B{N}$, and any two group-like elements
  $\s_1,\s_2\in \C{G}(\C{H})$. In particular, the $\B{N}$-graded
  vector space $\CH^\bullet(\C{H}, k)$ forms a graded $*$-subalgebra.
\end{proposition}

\begin{proof}
  Let us recall the coface maps of the cosimplicial module
  ${\bf C}^\bullet(\C{H},{}^{\s_1} k) = \bigoplus_{n\in\B{N}} \C{H}^{\otimes
    n}$ by
  \begin{align*}
    d_0(1) = & 1 - \s_1,\\
    d_i(h^1\otimes\cdots\otimes h^n) = &
    \begin{cases}
      1\otimes h^1\otimes\cdots\otimes h^n  & \text{ if } i=0,\\
      \cdots\otimes h^{i-1}\otimes {h^i}\ps{1}\otimes {h^i}\ps{2}\otimes h^{i+1}\otimes\cdots & \text{ if } 1\leq i\leq n,\\
      h^1\otimes\cdots\otimes h^n\otimes\alpha & \text{ if } i=n+1.
    \end{cases}
  \end{align*}
  The product on the chain level is defined as
  \begin{equation*}
    \Psi\smile \Phi := \Psi\otimes (\s_1\tr\Phi) \in \mathbf{C}^{p+q}(\C{H},{}^{\s_1\s_2} k)
  \end{equation*}
  for any $\Psi\in \mathbf{C}^p(\C{H},{}^{\s_1} k)$ and
 any $\Phi\in \mathbf{C}^q(\C{H},{}^{\s_2} k)$, where $\tr$ denotes the
  left diagonal action of $\C{H}$ on $\C{H}^{\otimes n}$. We thus note that
  \begin{align*}
    d_{p+q}(\Psi\smile\Phi)
    = & \sum_{i=0}^p(-1)^i d_i(\Psi)\otimes(\s_1\tr\Phi)\\
      & + (-1)^{p+1}\Psi\otimes\s_1\otimes (\s_1\tr\Phi)
        + (-1)^{p+2}\Psi\otimes\s_1\otimes (\s_1\tr\Phi)\\
      & + \sum_{i=1}^q(-1)^{p+i}\Psi\otimes(\s_1\tr d_i(\Phi))\\
    = & d_p(\Psi)\smile \Phi + (-1)^p \Psi\smile d_q(\Phi),
  \end{align*}
  \ie  $\bigoplus_{\s\in\C{G}(\C{H})} \CH^\bullet(\C{H},{}^\s k)$
  is a differential graded algebra with unit $1\in \mathbf{C}^0(\C{H}, k)$.  In particular,
  $\CH^\bullet(\C{H}, k)$ is a differential graded subalgebra.

  In Lemma~\ref{HopfHochschildCompatibleWithStar} we showed that
  $*$-structure defined in \eqref{star-Hopf-cyclic} is compatible with
  the Hochschild differentials. We now show that it is compatible with
  the cup product structure above.  For that we observe
  \begin{align*}
    w_{p+q} & ((h^1\otimes\cdots\otimes h^p)\smile(h^{p+1}\otimes\cdots\otimes h^{p+q}))\\
    = &  w_{p+q}(h^1\otimes\cdots\otimes h^p\otimes \s_1 h^{p+1}\otimes\cdots\otimes \s_1 h^{p+q})\\
    = & \s_2\s_1 S^{-1}({h^{p+q}}^*\alpha)\otimes\cdots\otimes \s_2\s_1 S^{-1}({h^{p+1}}^*\s_1)\otimes \s_2\s_1 S^{-1}({h^p}^*)\otimes\cdots\otimes \s_2\s_1 S^{-1}({h^1}^*)\\
    = & \s_2 S^{-1}({h^{p+q}}^*)\otimes\cdots\otimes \beta S^{-1}({h^{p+1}}^*)\otimes \s_2\s_1 S^{-1}({h^p}^*)\otimes\cdots\otimes \s_2\s_1 S^{-1}({h^1}^*)\\    
    = & w_q(h^{p+1}\otimes\cdots\otimes h^{p+q})\smile \omega_p(h^1\otimes\cdots\otimes h^p).
  \end{align*}
\end{proof}

We note that the dg-algebra structure on the sum
$\bigoplus_{\s\in\C{G}(\C{H})}\CH^\bullet(\C{H},{}^\s k)$ works even
in the case $\C{G}(\C{H})$ is not abelian.  However, we need
$\C{G}(\C{H})$ to be an abelian group for the $*$-structure to work.

The definition above is in fact a simplified version
of the following.  The collection
\begin{equation*}
  \bigoplus_{\s_1,\s_2\in\C{G}(\C{H})}\CB^\bullet(k^{\s_1},\C{H},{}^{\s_2} k)
\end{equation*}
form a differential graded category where the set of objects is
$\C{G}(\C{H})$, and the set of morphisms are defined as
$\Hom(\s_2,\s_1) := \CB^\bullet(k^{\s_1},\C{H},{}^{\s_2} k)$.  This
category is monoidal where the product comes from the product in
$\C{H}$.  Then we put an equivalence relation on the set of all
morphisms declaring two morphisms $\s_2\xrightarrow{\Psi}\s_1$ and
$\s_2'\xrightarrow{\Psi'}\s_1'$ to be equivalent if the morphisms
$1\xrightarrow{\Psi\tl \s_2^{-1}}\s_1\s_2^{-1}$ and
$1\xrightarrow{\Psi'\tl (\s_2')^{-1}}\s_1'(\s_2')^{-1}$ are identical.
Now, the set of equivalence classes of morphisms is the algebra we
defined above.

\section{Complexified QUE algebras and Their Hopf-Dihedral Cohomology}\label{Sect4}

\subsection{Drinfeld-Jimbo QUE algebras}

Let $\G{g}$ be a semi-simple Lie algebra of rank $\ell$,
$\alpha_1,\ldots,\alpha_n$ an ordered sequence of the simple roots,
and $a_{ij} = (\alpha_i,\alpha_j)$ the corresponding Cartan matrix.
The Drinfeld-Jimbo quantum enveloping algebra (QUE algebra)
$U_q(\G{g})$ is the Hopf algebra with $4\ell$ generators
$K_i, K_i^{-1}, E_i, F_i$, $1\leq i\leq \ell$, subject to the
relations
\begin{align*}
& K_iK_j =  K_jK_i,        \quad       K_iK_i^{-1} =  K_i^{-1}K_i=1,\\
& K_iE_jK_i^{-1} = q_i^{a_{ij}}E_j, \quad  K_iF_jK_i^{-1}  = q_i^{-a_{ij}}F_j,\\
& E_iF_j-F_jE_i = \delta_{ij}\frac{K_i-K_i^{-1}}{q_i-q_i^{-1}},\\
& \sum_{r=0}^{1-a_{ij}}(-1)^r\left[\begin{array}{c}
                                              1-a_{ij} \\
                                              r
                                            \end{array}
\right]_{q_i}E_i^{1-a_{ij}-r}E_jE_i^r=0, \quad i\neq j,\\
& \sum_{r=0}^{1-a_{ij}}(-1)^r\left[\begin{array}{c}
                                              1-a_{ij} \\
                                              r
                                            \end{array}
\right]_{q_i}F_i^{1-a_{ij}-r}F_jF_i^r=0, \quad i\neq j,
\end{align*}
where
\begin{equation*}
\left[ \begin{matrix} n \\ r \end{matrix} \right]_q 
= \frac{(n)_q\,!}{(r)_q\,!\,\,(n-r)_q\,!},\qquad (n)_q:=\frac{q^n-q^{-n}}{q-q^{-1}}.
\end{equation*}
The rest of the Hopf algebra structure of $U_q(\G{g})$ is given by
\begin{align*}
& \Delta(K_i)=K_i\otimes K_i,\quad \Delta(K_i^{-1})=K_i^{-1}\otimes K_i^{-1} \\
& \Delta(E_i)=E_i\otimes K_i + 1\otimes E_i,\quad \Delta(F_j)=F_j\otimes 1 + K_j^{-1}\otimes F_j \\
& \varepsilon(K_i)=1,\quad \varepsilon(E_i)=\varepsilon(F_i)=0\\
& S(K_i)=K_i^{-1},\quad S(E_i)=-E_iK_i^{-1},\quad S(F_i)=-K_iF_i.
\end{align*}

\subsection{Complexified Drinfeld-Jimbo QUE algebras}

Let $\G{g}$ be a semi-simple Lie algebra of rank $\ell$, and
$q^4\neq 0,1$. The complexified QUE algebra $\G{U}_q(\G{g})$ of
$\G{g}$ is the algebra generated by $4\ell$ generators
$K_i, K_i^{-1}, E_i, F_i$, $i=1,\ldots,\ell$, subject to the relation
we gave above, except the following:
\begin{align*}
& K_iE_jK_i^{-1} = q_i^{a_{ij}/2}E_j, \quad  K_iF_jK_i^{-1}  = q_i^{-a_{ij}/2}F_j,\\
& E_iF_j-F_jE_i = \delta_{ij}\frac{K_i^2-K_i^{-2}}{q_i-q_i^{-1}}.
\end{align*}
The rest of the Hopf algebra structure is defined as
\begin{align*}
& \Delta(K_i)=K_i\otimes K_i,\quad \Delta(K_i^{-1})=K_i^{-1}\otimes K_i^{-1},\\
& \Delta(E_i)=E_i\otimes K_i + K_i^{-1}\otimes E_i,\quad \Delta(F_j)=F_j\otimes K_j + K_j^{-1}\otimes F_j, \\
& \varepsilon(K_i)=1,\quad \varepsilon(E_i)=\varepsilon(F_i)=0,\\
& S(K_i)=K_i^{-1},\quad S(E_i)=-q_iE_i,\quad S(F_i)=-q_i^{-1}F_i.
\end{align*}

The Hopf algebra $\G{U}_q(\G{g})$ is a Hopf
$\ast$-algebra by
\begin{equation*}
  K_i^* = K_i \qquad (K_i^{-1})^* = K_i^{-1} \qquad 
  E_i^* = F_i \qquad F_i^* = E_i.
\end{equation*}
Indeed,
\begin{align*}
  \Delta(E_i^*) 
  = &  E_i\ps{1}^*\otimes E_i\ps{2}^* 
      =  E_i^*\otimes K_i^* + (K_i^{-1})^*\otimes E_i^*
   = F_i\otimes K_i + K_i^{-1}\otimes F_i  =\Delta(F_i)  \\
 \Delta(F_i^*) =  & F_i\ps{1}^*\otimes F_i\ps{2}^* 
      =  F_i^*\otimes K_i^* + (K_i^{-1})^*\otimes F_i^* =E_i\otimes K_i + K_i^{-1}\otimes E_i
       = \Delta(E_i).
\end{align*}
It is straightforward to check the condition for the group-like elements $K_i$ and
$K_i^{-1}$.  As for the counit, we have
\begin{align*}
  \varepsilon(E_i) = \varepsilon(E_i^*) = \varepsilon(F_i) = \varepsilon(F_i^*) = 0    
\end{align*}
and
\begin{align*}
  \varepsilon(K_i) = \varepsilon(K_i^*) = \varepsilon(K_i^{-1}) = \varepsilon((K_i^{-1})^*) = 1,
\end{align*}
and finally for the antipodes we see that
\begin{align*}
  S(E_i^*) = & S(F_i) = - q_i^{-1} F_i = S^{-1}(E_i)^*,\\
  S(F_j^*) = & S(E_j) = - q_j E_j = S^{-1}(F_i)^*.
\end{align*}

We took the definition of the complexified QUE algebra
$\G{U}_q(\G{g})$ and its $*$-structure from
\cite{KlimSchm-book}, in which the authors use the notation {\em
  \u{U}}$_q(\G{g})$. We also note from \cite{KlimSchm-book} that the
Hopf algebras $U_q(\G{g})$ and $\G{U}_q(\G{g})$ are not isomorphic in
general.  It also follows directly from the new commutator relations
that $(K^2_{2\rho},\ve)$ is a MPI over the Hopf algebra
$\G{U}_q(\G{g})$, where
$K^2_{2\rho} = K^2_1\cdots K^2_\ell \in \G{U}_q(\G{g})$.

In view of Theorem \ref{Hopf-cyclic-stable-coeff} and Proposition \ref{star-product} we conclude the following.

\begin{corollary}
  Let $\G{g}$ be a semi-simple Lie algebra of rank $\ell$.  Then the
  Hopf-cyclic cohomology
  $HC^\bullet(\G{U}_q(\G{g}),{}^{K^\mathbf{m}}k)$ of the complexified
  QUE algebra $\G{U}_q(\G{g})$ of $\G{g}$ carries a nontrivial
  $*$-structure, and thus the Hopf-dihedral cohomologies
  $HC^\bullet_\pm(\G{U}_q(\G{g}), {}^{K^\mathbf{m}}k)$ of
  $\G{U}_q(\G{g})$ are well-defined for any
  $\mathbf{m}\in\B{N}^{\times \ell}$.
\end{corollary}

\subsection{Complexified QUE algebra $\G{U}_q(sl_2)$}

The Hopf $*$-algebra $\G{U}_q(sl_2)$ is the algebra generated by $E,F,K$ and $K^{-1}$ subject to the relations
\begin{align*}
  KK^{-1} = & 1 & KEK^{-1} = & q E  & KFK^{-1} = & q^{-1} F & [E,F] = \frac{K^2-K^{-2}}{q-q^{-1}}.
\end{align*}
The $*$-structure is defined by
\begin{equation*}
  K^* = K  \quad\text{ and }\quad E^* =  F.
\end{equation*}
The rest of the Hopf $*$-algebra structure is given by
\begin{align*}
&  \Delta(K) = K\otimes K, \quad 
  \Delta(E) = E\otimes K + K^{-1} \otimes E, \quad   \Delta(F) = F\otimes K + K^{-1}\otimes F  \\
  &  \varepsilon(K) = \varepsilon(K^{-1}) = 1, \qquad \varepsilon(E) =  \varepsilon(F) = 0, \\
  & S(K^{\pm 1}) =  K^{\mp 1}, \qquad  S(E) = - q E, \qquad S(F) = - q^{-1} F.
\end{align*}

We finally note that $(K^2,\ve)$ is a MPI for the Hopf algebra $\G{U}_q(sl_2)$.

\subsection{Hopf-dihedral cohomology of $\G{U}_q(sl_2)$}

Considering the Hopf subalgebra $T = k[K,K^{-1}]$ of $\G{U}_q(sl_2)$
generated by $K^m$, $m\in\B{Z}$, there is a canonical
coextension of coalgebras of the form $\pi_T\colon \G{U}_q(sl_2)\to T$
given by
\begin{align*}
  \pi_T(E^uF^vK^w) = 
  \begin{cases}
    K^w & \text{ if } u+v=0,\\
    0   & \text{ otherwise},
  \end{cases}
\end{align*}
see also \cite[Sect. 4]{KaygSutl14}.

Since the projection $\pi_T$ is a map of coalgebras, we can consider
each basis element $E^uF^vK^w\in \G{U}_q(sl_2)$ as a 1-dimensional
$T$-comodule.  Then the spectral sequence associated to the
coextension $\pi_T\colon \G{U}_q(sl_2)\to T$ (see
\cite[Sect. 3]{KaygSutl14}) collapses onto the $r=0$ line as
\begin{align*}
  E^{0,0}_1 = & \begin{cases} \left<1\otimes 1\right> & \text{if  } m=0, \\
0 & \text{otherwise},  \end{cases}\\ 
  E^{r,0}_1 = & \bigoplus_{\mathbf{u},\mathbf{v}} 
     \left<1\otimes E^{u_1}F^{v_1}K^{w_1}
      \otimes \cdots\otimes E^{u_r}F^{v_r}K^{w_r}\otimes K^m\right>,
\end{align*}
where $\mathbf{u}$ and $\mathbf{v}$ run over the vectors of dimension
$p$ of non-negative integers satisfying the recursive formula
\begin{equation*}
  w_{i+1} = w_i + (u_{i+1} + v_{i+1}) + (u_i + v_i)
\end{equation*}
with the initial condition $u_0 = v_0 = w_0 = 0$, and the boundary
conditions $u_{r+1} = v_{r+1} = 0$ and $w_{r+1}=m$. Hence,
\begin{equation*}
  w_i = (u_i+v_i) + 2\sum_{j=1}^{i-1} (u_j+v_j) 
\end{equation*}
for $1\leq i\leq r+1$, and we note that $m$ must be a positive even
number.

Furthermore, $E^{r,s}_2 = E^{r,s}_\infty = HH^{r+s}(\G{U}_q(sl_2), {}^{K^{m}}k)$, since the $E_1$-term collapses onto the $r=0$ row, and we see that the collection $\bigoplus_{r,m\geq 0}HH^r(\G{U}_q(sl_2), {}^{K^{m}}k)$ of Hochschild cohomology groups form an algebra generated by
\begin{equation*}
HH^1(\G{U}_q(sl_2), {}^{K^{m}}k) = \left<(EK)^u(KF)^v|\ u+v = m\right>.
\end{equation*}

Let us {\bf formally} write
\begin{equation*}
HH^{\frac{1}{2}}(\G{U}_q(sl_2), {}^{K^{u}}k) = \left<(EK)^u, (KF)^u\right>.
\end{equation*}
Then one can easily see that
\begin{equation*}
HH^1(\G{U}_q(sl_2), {}^{K^{m}}k) = \bigoplus_{u+v=m} HH^{\frac{1}{2}}(\G{U}_q(sl_2), {}^{K^{u}}k) \ot 
              HH^{\frac{1}{2}}(\G{U}_q(sl_2), {}^{K^{v}}k).
\end{equation*}
Combining with Proposition \ref{star-product} we have the following result.

\begin{theorem}
The collection
  \begin{equation}
    \bigoplus_{p\in \B{N}\left<\frac{1}{2}\right>}\bigoplus_{m\in\B{N}}
    HH^{p}(\G{U}_q(sl_2), {}^{K^{m}}k)
  \end{equation}
of Hochschild cohomology groups form a $*$-algebra generated by
  $HH^{\frac{1}{2}}(\G{U}_q(sl_2), {}^{K^{m}}k)$ for
  $m\in\B{N}$.
\end{theorem}

Furthermore, recalling the Hopf-cyclic cohomology with generalized coefficients from Proposition \ref{CohomologicalPairing}, we conclude the following.

\begin{corollary}
The Hopf-cyclic cohomology, with generalized coefficients, of the complexified QUE algebra
  $\G{U}_q(sl_2)$ is given by
  \begin{align*}
    HC_{\G{U}_q(sl_2)}^{p+1} & (\G{U}_q(sl_2), {}^{K^{m}}k)\\ = &
    \begin{cases}
      \left<S^{p/2}((EK)^u(KF)^v)|\ u+v=m\right> & \text{ if $p$ is even},\\
      0 & \text{otherwise},
    \end{cases}
  \end{align*}
  where $S$ here denotes the iteration with the Bott
  element in cyclic cohomology. As a result, the Hopf-dihedral
  cohomology of $\G{U}_q(sl_2)$ is 
  \begin{align*}
    HC_{\G{U}_q(sl_2),\pm}^{p+1} & (\G{U}_q(sl_2), {}^{K^{m}}k)\\ = &  
    \begin{cases}
      \left<S^{p/2}((EK)^u(KF)^v\pm (EK)^v(KF)^u)|\ u+v=m\right>& \text{ if $p$ is even},\\
      0 & \text{otherwise}.
    \end{cases}
  \end{align*}
\end{corollary}

\subsection{Hopf-dihedral cohomology of complexified QUE algebras}

Considering the Hopf subalgebra
$T^\ell = k[K_1^{\pm 1},\ldots,K_\ell^{\pm 1}]$ of $\G{U}_q(\G{g})$,
generated by $K^m_i$ where $i=1,\ldots,\ell$ and $m\in\B{Z}$, there
is a canonical coextension of coalgebras of the form
$\pi_{T^\ell}\colon \G{U}_q(\G{g})\to T^\ell$ given by
\begin{align*}
  \pi_{T^\ell}(E_1^{u_1}\cdots & E_\ell^{u_\ell}F_1^{v_1}\cdots F_\ell^{v_\ell}K_1^{w_1}\cdots K_\ell^{w_\ell})\\ = &
  \begin{cases}
    K_1^{w_1}\cdots K_\ell^{w_\ell} & \text{ if } u_1+\cdots+u_\ell+v_1+\cdots+v_\ell=0,\\
    0   & \text{ otherwise}.
  \end{cases}
\end{align*}
In order to simplify the notation, we use multi-indices
$E^{\mathbf{u}}F^{\mathbf{v}}K^{\mathbf w}$ for any monomial of the
form
$E_1^{u_1}\cdots E_\ell^{u_\ell}F_1^{v_1}\cdots
F_\ell^{v_\ell}K_1^{w_1}\cdots K_\ell^{w_\ell}$
where $\mathbf{u}$, $\mathbf{v}$ and $\mathbf{w}$ are elements in
$\B{N}^{\times \ell}$.


\begin{theorem}
For any complexified QUE algebra $\G{U}_q(\G{g})$, the Hochschild
  cohomology groups 
  \begin{align*}
    \bigoplus_{p\in\B{N}\left<\frac{1}{2}\right>}\bigoplus_{\mathbf{m}\in\B{N}^{\times\ell}} HH^{p}(\G{U}_q(\G{g}), {}^{K^{\bf m}}k)
  \end{align*}
  form a $*$-algebra generated by
  $HH^{\frac{1}{2}}(\G{U}_q(\G{g}), {}^{K^{\bf m}}k) = \left<
    (E_j K_j)^u\pm (K_j F_j)^u\right>$.
\end{theorem}

As a result, the Hopf-cyclic cohomology with generalized coefficients is obtained as follows.

\begin{corollary}
  For any complexified QUE algebra $\G{U}_q(\G{g})$, the Hopf-cyclic
  cohomology of $\G{U}_q(\G{g})$ is given by
  \begin{align*}
    HC_{\G{U}_q(\G{g})}^{p+1} & (\G{U}_q(\G{g}), {}^{K^{\bf m}}k)\\ = & 
    \begin{cases}
      \left<S^{p/2}((EK)^\mathbf{u}(KF)^\mathbf{v})|\ \mathbf{u}+\mathbf{v}=\mathbf{m}\right> & \text{ if $p$ is even},\\
      0 & \text{otherwise},
    \end{cases}
  \end{align*}
  where $S$ is the iteration with the Bott
  element, and hence the Hopf-dihedral
  cohomology of $\G{U}_q(\G{g})$ is 
  \begin{align*}
    HC_{\G{U}_q(\G{g}),\pm}^{p+1} & (\G{U}_q(\G{g}), {}^{K^{\bf m}}k)\\ = &
    \begin{cases}
      \left<S^{p/2}((EK)^\mathbf{u}(KF)^\mathbf{v}\pm (EK)^\mathbf{v}(KF)^\mathbf{u})|\ \mathbf{u}+\mathbf{v}=\mathbf{m}\right>& \text{ if $p$ is even},\\
      0 & \text{otherwise}.
    \end{cases}
  \end{align*}
\end{corollary}

\section{Hopf-Dihedral Homology}\label{Sect5}

In this section we develop a $\ast$-structure on Hopf-cyclic
homology, \cite{KhalRang02} and see also \cite{Tail01}, the cyclic dual of Hopf-cyclic cohomology.

\subsection{The dual cyclic theory}

Let $\C{H}$ be a Hopf algebra with a modular pair in involution
$(\sigma,\delta)$.  It is shown in \cite{KhalRang02} that the graded space
\begin{equation}\label{Hopf-cyclic-homology-complex}
\CC_\bullet(\C{H};\s,\d) = \bigoplus_{n\geq 0} \C{H}^{\ot\,n}
\end{equation}
has the structure of a cyclic module via the face operators
$\p_i: \CC_n(\C{H},\d,\s) \lra \CC_{n-1}(\C{H},\d,\s)$ defined for
$i\leq n$ by
\begin{align*}
\partial_i(h^1\odots h^n) = & 
  \begin{cases}
    \ve(h^1)h^2\odots h^n           & \text{ if } i=0,\\
    h^1\odots h^ih^{i+1} \odots h^n  & \text{ if } 1 \leq i \leq n-1, \\ 
    \d(h^n)h^1\odots h^{n-1}         & \text{ if } i=n.
  \end{cases}
\end{align*}
The degeneracies, on the other hand, are of the form
$\s_j: \CC_n(\C{H},\d,\s) \lra \CC_{n+1}(\C{H},\d,\s)$ defined for
$0\leq i\leq n$ as 
\begin{align*}
  \s_j(h^1\odots h^n) = & 
  \begin{cases}
    1 \ot h^1\odots h^n                      & \text{ if } j=0, \\
    h^1\odots h^j\ot 1 \ot h^{j+1} \odots h^n & \text{ if } \quad 1 \leq j \leq n-1, \\ 
    h^1\odots h^{n} \ot 1                    & \text{ if } j=n.
  \end{cases}
\end{align*}
Finally, we have the cyclic operators
$\tau_n: \CC_n(\C{H},\d,\s) \lra \CC_{n}(\C{H},\d,\s)$ given by
\begin{align*}
  \tau_n(h^1\odots h^n) = \d(h^n_{(2)})S_\s(h^1_{(1)}\ldots h^n_{(1)})\ot h^1_{(2)}\odots h^{n-1}_{(2)},
\end{align*}
where $S_\s(h) := \s S(h)$ for any $h\in \C{H}$. The cyclic homology of the complex \eqref{Hopf-cyclic-homology-complex} is called the Hopf-cyclic homology of the Hopf algebra $\C{H}$, and is denoted by $HC_\bullet(\C{H};\s,\d)$.

\subsection{Comodule algebras, invariant traces, and Hopf-cyclic homology}

Let $\C{A}$ be a right $\C{H}$-comodule algebra by
$\nb:\C{A} \lra \C{A} \ot \C{H}$, where we write
$\nb(a)=a\ns{0}\ot a\ns{1}$, that is,
\begin{equation*}
  \nb(ab) = \nb(a)\nb(b), \qquad \nb(1) = 1 \ot 1.
\end{equation*}
Let also $\C{A}$ be equipped with a $\s$-invariant
$\d$-trace, that is, $Tr:\C{A} \lra k$ satisfying
\begin{equation*}
  Tr(a\ns{0})a\ns{1} = Tr(a)\s, \qquad Tr(ab) = Tr(ba\ns{0})\d(a\ns{1}).
\end{equation*} 
In the presence of these conditions we have $\gamma:\CC_n(\C{A}) \lra \CC_n(\C{H},\d,\s)$ a morphism of cyclic modules given by
\begin{align}\label{char-map-Hopf-homology}
  \gamma(a_0\odots a_n) := Tr(a_0a_1\ns{0}\ldots a_n\ns{0}) a_1\ns{1} \odots a_n\ns{1}.
\end{align}
Then, just as in the previous section, \eqref{char-map-Hopf-homology} induces the cyclic module structure on $\bigoplus_{n\geq 0} \C{H}^{\otimes n+1}$ as defined in \cite{KhalRang02}.

In particular, a Hopf algebra $\C{H}$ is a $\C{H}$-comodule algebra
via its comultiplication. This way one can compare the Hopf-cyclic
homology of $\C{H}$ with the algebra cyclic homology of $\C{H}$
regarding it as an algebra. To this end, it is noted in
\cite[Prop. 3.2]{KhalRang02} that if $\C{H}$ is a Hopf algebra with a
group-like $\s \in \C{H}$ such that $S_\s^2=\Id$, then
$\t:\CC_n(\C{H},\ve,\s) \lra \CC_n(\C{H})$, which is defined as
\begin{equation}\label{char-map-Hopf-homology-teta}
 \t(h_1\odots h^n) := S_\s(h^1\ps{1}h^2\ps{1} \ldots h^n\ps{1}) \ot h^1\ps{2} \odots h^n\ps{2},
\end{equation}
is a map of cyclic modules. If furthermore, $\C{H}$ is equipped with a $\s$-invariant trace $Tr:\C{H}\lra k$ such that $Tr(\s)$ is invertible in $k$, then it follows from $\gamma\circ \t = Tr(\s)\Id$ that $HC_\bullet(\C{H},\ve,\s)$ is a direct summand of $HC_\bullet(\C{H})$, \cite[Thm. 3.1]{KhalRang02}.

We finally record here that setting
$H_\bullet(\C{H}) := {\rm Tor}_\bullet^\C{H}(k,k)$ for a cocommutative
Hopf algebra $\C{H}$, it follows from \cite[Thm. 4.1]{KhalRang02} that
\begin{equation*}
HC_n(\C{H},\ve,1) = \bigoplus_{i\geq 0}H_{n-2i}(\C{H})
\end{equation*}
which proven by Karoubi~\cite{Karo83-III} in case $\C{H}=kG$.

\subsection{The dihedral structure}

Let us now assume that $\C{H}$ is a Hopf $\ast$-coalgebra, \ie a $\ast$-coalgebra
\begin{equation*}
\D(h^*)=h\ps{2}^* \ot h\ps{1}^*, \qquad \ve(h^*)=\wbar{\ve(h)}
\end{equation*}
such that
\begin{equation*}
(hg)^* = h^*g^*, \qquad 1^*=1, \qquad (S \circ \ast)^2 = \Id,
\end{equation*}
and that $\C{A}$ a $\C{H}$-comodule $\ast$-algebra, \ie 
\begin{equation*}
\nb(a^*) = {a\ns{0}}^* \ot S(a\ns{1})^* = {a\ns{0}}^* \ot S^{-1}({a\ns{1}}^*),\qquad \forall\, a\in \C{A}.
\end{equation*}
We also assume that $\d(h^*)=\d(h)$ and that $\s^* = \s^{-1}$. In
order to obtain a $*$-structure on
$\CC_\bullet(\C{H},\d,\s)$ we use \eqref{char-map-Hopf-homology} to
transfer the one on $\CC_\bullet(\C{H})$.  To this end we introduce
cyclic operators $\tau_n: \CC_n(\C{H},\d,\s) \lra \CC_n(\C{H},\d,\s)$
by
\begin{equation}\label{dihedral-tau-Hopf}
 \tau_n(h^1 \odots h^n) = \d(h^n\ps{2})S_\s(h^1\ps{1}\ldots h^n\ps{1})\ot h^1\ps{2}\odots h^{n-1}\ps{2},
\end{equation}
and involution operators $\om_n: \CC_n(\C{H},\d,\s) \lra \CC_n(\C{H},\d,\s)$ by
\begin{align}\label{dihedral-om-Hopf}
  \om_n(h^1 \odots h^n)
  = & \d({h^1\ps{1}}^* \ldots {h^n\ps{1}}^*)S^{-1}({h^n\ps{2}}^*) \odots S^{-1}({h^1\ps{2}}^*) \\\notag
  = & S^{-1}({h^n}^*\ps{1}) \odots S^{-1}({h^1}^*\ps{1})\d({h^1}^*\ps{2} \ldots {h^n}^*\ps{2}).
\end{align}
We then leave it to the reader to verify that we thus get a dihedral module structure on
$\bigoplus_{n\geq 0}\C{H}^{\otimes n+1}$.

We next record here that \cite[Thm. 3.1]{KhalRang02} also extends to the
dihedral setting. We will make use of the fact that if $\C{H}$ is a
Hopf $\ast$-coalgebra, then $\C{H}$ is a $\ast$-algebra via
$\ast\circ S$. If, in particular, $\C{H}$ is cocommutative, then we
may take $\ast = \Id$ for the Hopf $\ast$-coalgebra structure.

We denote the (dihedral) homology of this dihedral module by
$HC^\pm_*(\C{H};\s,\d)$, and call it the Hopf-dihedral homology.  

\begin{proposition}\label{prop-dihedral-char-map}
The characteristic homomorphism \eqref{char-map-Hopf-homology}
  defines a morphism of dihedral modules of the form
  $\gamma\colon \CC^\pm_\bullet(\C{A})\to \CC^\pm_\bullet(\C{H};\s,\d)$.
\end{proposition}

\begin{proof}
  It follows already from \cite[Prop. 3.1]{KhalRang02} that
  $\tau_n\gamma = \gamma\tau_n$ on $\CC_n(\C{A})$. It then takes a straightforward calculation to check that \eqref{char-map-Hopf-homology} is compatible with $\omega_n$.
\end{proof}

\begin{theorem}\label{direct-summand}
Let $\C{H}$ be a Hopf algebra equipped with a modular pair $(\ve,\s)$ in involution, and a $\s$-invariant trace $Tr:\C{H}\lra k$ such that $Tr(\s) \in k$ is invertible. Then $HC^\pm_n(\C{H}; \s,\ve)$ is a direct summand of $HC^\pm_n(\C{H})$.
\end{theorem}

\begin{proof}
We have already observed in Proposition \ref{prop-dihedral-char-map} that
\eqref{char-map-Hopf-homology} is a map
  of dihedral modules. Furthermore, it is checked in
  \cite[Thm. 3.1]{KhalRang02} that $\gamma\circ\t = Tr(\s)\Id$. On the other hand, it is shown in
  \cite[Prop. 3.2]{KhalRang02} that \eqref{char-map-Hopf-homology-teta} is a cyclic map. We are then left to observe that
\begin{align*}
 \om_n\t(h^1\odots h^n) 
 = & \om_n(S_\s(h^1\ps{1}\ldots h^n\ps{1}) \ot h^1\ps{2} \odots h^n\ps{2})\\
 = & (S^2(h^1\ps{1}\ldots h^n\ps{1})\s^{-1})^* \ot S({h^n\ps{2}})^* \odots S({h^1\ps{2}})^* \\
 = & S^{-2}({h^1\ps{1}}^*\ldots {h^n\ps{1}}^*)\s \ot S^{-1}({h^n\ps{2}}^*) \odots S^{-1}({h^1\ps{2}}^*)\\
 = & \s{h^1\ps{1}}^*\ldots {h^n\ps{1}}^* \ot S^{-1}({h^n\ps{2}}^*) \odots S^{-1}({h^1\ps{2}}^*)\\
 = & S_\s(S^{-1}({h^n\ps{1}}^*)\ldots S^{-1}({h^1\ps{1}}^*)) \ot S^{-1}({h^n\ps{2}}^*) \odots S^{-1}({h^1\ps{2}}^*)\\
 = & \t(S^{-1}({h^n}^*) \odots S^{-1}({h^1}^*))\\
 = & \t\om_n(h^1 \odots h^n),
\end{align*}
that is \eqref{char-map-Hopf-homology-teta} is a map of dihedral modules.
\end{proof}

\subsection{Hopf-dihedral homology of cocommutative Hopf algebras}

In this subsection, we extend \cite[Sect. 4]{KhalRang02} to the dihedral setting, and we compute the Hopf-dihedral homology of cocommutative Hopf algebras. To this end, we first recall the path space $E\C{H}_\bullet$ of the Hopf algebra $\C{H}$, which is defined on the graded module level as $E\C{H}_n:=\C{H}^{\ot\,n+1}$, whose simplicial structure is given by
\begin{align*}
  \p_i(h^0\odots h^n) = &
  \begin{cases}
    h^0\odots h^ih^{i+1}\odots h^n & \text{ if } \qquad 0\leq i\leq n-1,\\
    \d(h^n)h^0\odots h^{n-1}       & \text{ if } i=n,
  \end{cases}\\
  \s_j(h^0 \odots h^n) = &
    h^0 \odots h^j \ot 1\ot h^{j+1} \odots h^n, \qquad 0\leq j\leq n.
\end{align*}
The simplicial module $E\C{H}_\bullet$ is contractible, and is a resolution for $k$ via $\d:\C{H}\lra k$.

\begin{lemma}
If $\C{H}$ is a cocommutative Hopf algebra, then $E\C{H}_\bullet$ is a dihedral module by
\begin{equation}\label{dihedral-path-space-tau}
\tau_n(h^0 \odots h^n) = h^0h^1\ps{1}\ldots h^n\ps{1} \ot S(h^1\ps{2}\ldots h^n\ps{2}) \ot h^1\ps{3} \odots h^{n-1}\ps{3}
\end{equation}
and
\begin{equation}\label{dihedral-path-space-om}
\om_n(h^0 \odots h^n) = h^0{h^1\ps{1}}\ldots {h^n\ps{1}} \ot S^{-1}(h^n\ps{2}) \odots S^{-1}(h^1\ps{2}).
\end{equation}
\end{lemma}

\begin{proof}
It is checked in \cite[Lemma 4.2]{KhalRang02} that $E\C{H}_\bullet$ is a cyclic module via \eqref{dihedral-path-space-tau}. As a result, we have 
\begin{equation*}
\p_i\tau_n = \tau_{n-1}\p_{i-1}, \qquad \s_i\tau_n = \tau_{n+1}\s_{n-1}, \quad \qquad 1 \leq i \leq n
\end{equation*}
already. We next check that
\begin{align*}
  \om_n^2 & (h^0\odots h^n) \\
  = & h^0{h^1_{(1)}}\ldots {h^n_{(1)}}S^{-1}({h^n_{(2)}}) \ldots S^{-1}({h^1}_{(2)}) \ot S^{-1}(S^{-1}({h^1_{(2)}})) \odots S^{-1}(S^{-1}({h^n_{(2)}}))\\
  = & h^0 \ot S^{-2}(h^1)\odots S^{-2}(h^n)\\
  = & h^0\odots h^n,
\end{align*}
that is $\om_n^2 =\Id$. Moreover we have
\begin{align*}
  \p_i\om_n & (h^0\odots h^n)\\
  = & h^0{h^1_{(1)}}\ldots {h^n_{(1)}} \ot S^{-1}({h^n_{(2)}}) \odots S^{-1}({h^{n-i+1}_{(2)}})S^{-1}({h^{n-i}_{(2)}}) \odots S^{-1}({h^1_{(2)}}) \\
  = & \om_n\p_{n-i}(h^0\odots h^n),
\end{align*}
and 
\begin{align*}
  \s_j\om_n & (h^0\odots h^n) \\
  = & h^0{h^1_{(1)}}\ldots {h^n_{(1)}} \ot S^{-1}({h^n_{(2)}}) \odots S^{-1}({h^{n-j+1}_{(2)}})\ot 1 \ot S^{-1}({h^{n-j}_{(2)}}) \odots S^{-1}({h^1_{(2)}}) \\
  = & \om_n\s_{n-j}(h^0\odots h^n).
\end{align*}
Finally, we see that
\begin{align*}
  (\tau_n\om_n) & (h^0 \odots h^n)\\
  = & h^0h^1_{(1)}\ldots h^n_{(1)}S^{-1}({h^n_{(2)}}) \ldots S^{-1}({h^1_{(2)}}) \ot S(S^{-1}({h^n_{(3)}}) \ldots S^{-1}({h^1_{(3)}})) \ot S^{-1}({h^n_{(4)}}) \odots S^{-1}({h^2_{(4)}}) \\
  = & h^0 \ot h^1h^2_{(1)} \ldots h^n_{(1)} \ot S^{-1}({h^n_{(2)}}) \odots S^{-1}({h^2_{(2)}}),
\end{align*}
and hence
\begin{align*}
  (\tau_n\om_n)^2(h^0 \odots h^n)
  = & h^0 \ot h^1h^2_{(1)} \ldots h^n_{(1)}S^{-1}({h^n_{(2)}}) \ldots S^{-1}({h^2_{(2)}}) \ot S^{-2}(h^2_{(3)}) \odots S^{-2}(h^n_{(3)}) \\
  = & h^0 \odots h^n,
\end{align*}
that is $\tau_n\om_n\tau_n\om_n=\Id$, or equivalently $\tau_n\om_n = \om_n\tau_n^{-1}$.
\end{proof}

\begin{lemma}\label{lemma-pi-map}
  If $\C{H}$ is a cocommutative Hopf algebra, then
  $\pi:E\C{H}_n \lra \CC_n(\C{H}; 1,\ve)$ defined as
\begin{equation*}
  \pi(h^0\odots h^n) = \ve(h^0)h^1\odots h^n
\end{equation*}
is a map of dihedral modules.
\end{lemma}

\begin{proof}
As a result of \cite[Lemma 4.2]{KhalRang02} it suffices to observe
\begin{align*}
 \pi\om_n(h^0 \odots h^n)
  = & \ve(h^0)\ve({h^1_{(1)}}\ldots {h^n_{(1)}})S^{-1}(h^n_{(2)}) \odots S^{-1}(h^1_{(2)}) \\
  = & \ve(h^0)\om_n(h^1\odots h^n) \\
  = & \om_n\pi(h^0\odots h^n)
\end{align*}
as we wanted to show.
\end{proof}

\begin{theorem}\label{thm-dihedral-cocomm}
If $\C{H}$ is a cocommutative Hopf algebra, then for any $n\geq 0$
\begin{equation*}
HC^+_n(\C{H};1,\ve) = \bigoplus_{i\geq 0} H_{n-4i}(\C{H},k),
\quad\text{ and }\quad
HC^-_n(\C{H};1,\ve) = \bigoplus_{i\geq 0} H_{n-4i-2}(\C{H},k).
\end{equation*}
\end{theorem}

\begin{proof}
Let us first note that $E\C{H}_\bullet$ is a (left) $\C{H}$-module by $h \cdot (h^0 \odots h^n) = hh^0 \odots h^n$, and hence $\CC_\bullet(\C{H};1,\ve) = k\ot_\C{H} E\C{H}_\bullet$. Moreover, as a result of Lemma \ref{lemma-pi-map} we have $\CC^\pm_\bullet(\C{H};1,\ve) \cong k \ot_\C{H} \CC^\pm_\bullet(E\C{H}_\bullet)$ as bicomplexes where the latter is the dihedral bicomplex of the dihedral module $E\C{H}_\bullet$.

Since $E\C{H}_\bullet$ is contractible the vertical homology vanishes, and since $\ve:\C{H}\lra k$ is the contracting homotopy, the double complex $\CC^+_\bullet(E\C{H}_\bullet)$ is a (Cartan-Eilenberg) resolution for 
\begin{equation*}
k^+_\bullet\,:\quad k \longleftarrow 0 \longleftarrow 0 \longleftarrow 0 \longleftarrow k \longleftarrow 0 \longleftarrow 0 \longleftarrow 0 \longleftarrow \ldots
\end{equation*}
and $\CC^-_\bullet(E\C{H}_\bullet)$ is for 
\begin{align*}
k^-_\bullet\,:\quad 0 \longleftarrow 0 \longleftarrow k \longleftarrow 0 \longleftarrow 0 \longleftarrow 0 \longleftarrow k \longleftarrow 0 \longleftarrow 0 \longleftarrow \ldots
\end{align*}
As a result,
\begin{align*}
  HC^+_n(\C{H};1,\ve) 
  = & H_n({\rm Tot}(\CC^+(\C{H};1,\ve))) = H_n({\rm Tot}(k \ot_\C{H} \CC^+(E\C{H})))\\
  = & H_n(k \ot_\C{H} {\rm Tot}(\CC^+(E\C{H}))) = \mathbb{H}_n(\C{H},k^+_\bullet),
\end{align*}
and similarly
\begin{align*}
  HC^-_n(\C{H};1,\ve)
  = & H_n({\rm Tot}(\CC^-(\C{H};1,\ve))) = H_n({\rm Tot}(k \ot_\C{H} \CC^-(E\C{H}))) \\
  = & H_n(k \ot_\C{H} {\rm Tot}(\CC^-(E\C{H}))) = \mathbb{H}_n(\C{H},k^-_\bullet),
\end{align*}
where the latter objects are the hyperhomologies of the complexes
$k^+_\bullet$ and $k^-_\bullet$, respectively. Since the hyperhomology
is independent of the (Cartan-Eilenberg) resolution chosen, we can
replace the double complex $\CC^+_\bullet(E\C{H}_\bullet)$ with one
having zeros as the horizontal maps, and having resolutions of $k$ on
the zeroth (mod 4) column.  Similarly, we replace the double complex
$\CC^-_\bullet(E\C{H}_\bullet)$ with one having resolutions of $k$ on
the second (mod 4) column. Therefore, for the $+1$-eigenspace we get
\begin{align*}
\xymatrix{
  \vdots \ar[d] & \vdots \ar[d] & \vdots \ar[d] & \vdots \ar[d] & \vdots \ar[d] &    \\
H_1(\C{H},k) \ar[d]   & 0 \ar[l] \ar[d]   & 0 \ar[l] \ar[d]  & 0 \ar[l] \ar[d] & H_1(\C{H},k) \ar[l] \ar[d] & \ldots \ar[l] \\
H_0(\C{H},k)   & 0 \ar[l]   & 0 \ar[l]  & 0\ar[l]  & H_0(\C{H},k) \ar[l]  & \ldots \ar[l]
}
\end{align*}
and for the $-1$-eigenspace we get
\begin{align*}
\xymatrix{
  \vdots \ar[d] & \vdots \ar[d] & \vdots \ar[d] & \vdots \ar[d] & \vdots \ar[d] & \vdots \ar[d] & \vdots \ar[d] &    \\
0 \ar[d] & 0 \ar[d] \ar[l] & H_1(\C{H},k) \ar[d] \ar[l]  & 0 \ar[l] \ar[d]   & 0 \ar[l] \ar[d]  & 0 \ar[l] \ar[d] & H_1(\C{H},k) \ar[l] \ar[d] & \ldots \ar[l] \\
0 & 0 \ar[l] & H_0(\C{H},k) \ar[l]  & 0 \ar[l]   & 0 \ar[l]  & 0\ar[l]  & H_0(\C{H},k) \ar[l]  & \ldots \ar[l]
}
\end{align*}
on the $E_2$-page (and hence on the $E_\infty$-page, since the horizontal maps are zero) of the bicomplex computing the hyperhomology $\mathbb{H}_n(\C{H},k^\pm_\ast)$. The claim then follows.
\end{proof}

We finally relate the Hopf-dihedral homology of a group algebra to the
dihedral group homology.  We recall the basic facts from
\cite{KrasLapiSolo87}. Let $G$ be a (discrete) group, and $\C{A}$ be
the group algebra $kG$. Then the graded space given by
\begin{equation}\label{dihedral-group-hom-complex}
E\C{A}^\pm_n := {\rm Span}\{\,g_0\odots g_n \in \CC^\pm_n(\C{A}) \mid \, g_0\ldots g_n = 1\}
\end{equation}
form a dihedral submodule of the standard dihedral module
$\CC^\pm_\bullet(\C{A})$. The dihedral homology of this dihedral module is
called the dihedral group homology of the group $G$, and is denoted by
$HC_\bullet^\pm(G,k)$.

It follows at once from \cite[Cor. 4.4]{KrasLapiSolo87} and Theorem
\ref{thm-dihedral-cocomm} that the Hopf-dihedral homology of the Hopf
algebra $kG$ is isomorphic to the dihedral group homology of the group
$G$. In the following proposition we record also the explicit
isomorphism between the complexes.

\begin{proposition}\label{prop-group-dihedral}
For any (discrete) group $G$, $HC^\pm_\bullet(kG;1,\ve) \cong HC^\pm_\bullet(G,k)$.
\end{proposition}

\begin{proof}
  It follows from \eqref{dihedral-group-hom-complex} that
  $\t_n:\CC^\pm_n(kG;1,\ve) \lra EkG^\pm_n$ given by
  \begin{equation}\label{theta-group}
    \t_n(g_1\odots g_n):= (g_1\ldots g_n)^{-1} \ot g_1 \odots g_n
  \end{equation}
  is an isomorphism. It is also straightforward from the proof of
  Theorem \ref{direct-summand} that \eqref{theta-group} is a map of
  dihedral modules.
\end{proof}

As a result, we recover \cite[Thm. 4.5]{KrasLapiSolo87} from Theorem
\ref{direct-summand}. Moreover, from 
\begin{equation*}
HC^\a_n(G,k)\cong HC^\a_n(kG;1,\ve) = H_{n+\a-1}(G,k) \oplus H_{n+\a-5}(\C{H},k) \oplus H_{n+\a-9}(\C{H},k) \oplus \ldots \\
\end{equation*}
for $\a=\pm1$  we get the embeddings
\begin{equation}\label{embedding-i}
i_\ell:H_n(G,k) \lra HC^\a_{n+2\ell}(G,k), \qquad \a =(-1)^\ell.
\end{equation}

\section{Hopf-dihedral homology and $L$-theory}\label{Sect6}

\subsection{Hopf-dihedral Chern character}\label{Section:DihedralChernCharacter}

In this subsection we realize the Chern character as a homomorphism with values in the Hopf-dihedral homology. More precisely, given an involutive ring 
$\C{A}$, we will transfer the $KU$-groups of $\C{A}$ to the Hopf-dihedral homology of a Hopf algebra $\C{H}$ on which $\C{A}$ is a comodule algebra. Let us first recall the construction of the $KU$-functor from \cite{Bass73,KrasLapiSolo87}. For the $L$-functor, we refer the reader to \cite{Cortinas:1993a,Cortinas:1993b}, and the references therein.

Let $\C{A}$ be an algebra with involution, and $U^\epsilon_n(\C{A})$ the set of $2n\times 2n$ matrices with coefficients in $\C{A}$ preserving the quadratic form
\begin{equation*}
{}_\epsilon J = \left(\begin{array}{cc}
0 & I_n \\
\epsilon I_n & 0
\end{array}\right).
\end{equation*}
We note from \cite{Karoubi-book} that they are the matrices satisfying $M^\star M = MM^\star = I_{2n}$, where
\begin{equation*}
M^\star=\left(\begin{array}{cc}
D^\dagger & \epsilon B^\dagger \\
\epsilon C^\dagger & A^\dagger
\end{array}\right), \qquad M = \left(\begin{array}{cc}
A & B \\
C & D
\end{array}\right).
\end{equation*}
Here $\dagger$ denotes the conjugate transpose of an $n\times n$ matrix. We note also from \cite[Ex. 4.16]{Karoubi-book} that $U^1_n(\B{R}) = O(n,n)$, and $U^{-1}_n(\B{R}) = Sp(2n,\B{R})$. Now we set
\begin{equation*}
U^\epsilon(\C{A}) := \text{colim}_n\ U^\epsilon_n(\C{A}),
\end{equation*}
and following \cite{KrasLapiSolo87} we denote by $BU^\epsilon(\C{A})$
the classifying space of the group $U^\epsilon(\C{A})$. It follows
from \cite[Prop. 5.1(b) and Thm. 5.2]{Bass73} that the second derived
subgroup of $U^\epsilon(\C{A})$ is equal to its first derived
subgroup, i.e. it is quasi-perfect.  Thus, the
commutator subgroup of $U^\epsilon(\C{A})$ is a perfect normal
subgroup. As a result, the plus construction
\cite{Kerv69,Adams-book,Quill72} can be applied to the classifying
space $BU^\epsilon(\C{A})$ to obtain the space $BU^\epsilon(\C{A})^+$
such that
\begin{equation*}
H_\bullet(BU^\epsilon(\C{A})) \cong H_\bullet(BU^\epsilon(\C{A})^+)
\end{equation*}
and that 
\begin{equation*}
\pi_1(BU^\epsilon(\C{A})^+) = \frac{U^\epsilon(\C{A})}{[U^\epsilon(\C{A}),U^\epsilon(\C{A})]}.
\end{equation*}
Then, by definition \cite[Def. 4.17]{Karoubi-book}, the Hermitian algebraic $K$-theory of the algebra
$\C{A}$ is given by
\begin{equation*}
KU^\epsilon_n(\C{A}):=\pi_n(BU^\epsilon(\C{A})^+), \qquad n\geq 1,
\end{equation*}
and for $n=0$, $KU^\epsilon_0(\C{A})$ is the Grothendieck group of the category of non-degenerate quadratic forms over $\C{A}$. We now consider the composition 
\begin{align*}
  H_s(U_n^\epsilon(\C{A}),k) & \xrightarrow{i_{\ell}} 
  HC^\a_{s+2\ell}(U_n^\epsilon(\C{A}),k) \xrightarrow{j_{s+2\ell}}\\
   & HC^\a_{s+2\ell}(kU_n^\epsilon(\C{A})) \xrightarrow{J_\ell} 
     HC^\a_{s+2\ell}(M_{2n}(\C{A}))\xrightarrow{{\rm Tr}_\ell}
     HC^\a_{s+2\ell}(\C{A}),
\end{align*}
where $\a=(-1)^\ell$, and
\begin{enumerate}[(i)]
\item $i_{\ell}$ is the mapping given by \eqref{embedding-i},
\item $j_{s+2\ell}$ is the map induced by the inclusion of the
  dihedral group homology of a group into the dihedral homology of the
  group algebra of this group, 
\item $J_\ell$ is induced by the inclusion of $kU_n^\epsilon(\C{A})$
  into $M_{2n}(\C{A})$, 
\item ${\rm Tr}_\ell$ is induced by the trace map inducing the Morita
  isomorphism.
\end{enumerate}

Taking colimit over $n$ we obtain the mappings
$L^\ell:H_s(U^\epsilon(\C{A}),k) \lra HC^\a_{s+2\ell}(\C{A})$ for
$\ell\geq 0$ and $a=(-1)^\ell$.  Finally, the dihedral Chern character
\begin{equation}\label{dihedral-chern}
{\rm Ch}_s^\ell:KU_s^\epsilon(\C{A}) \lra HC^\a_s(\C{A}), \qquad s\geq 0, \,\,\a=(-1)^\ell,
\end{equation}
is defined to be the composition
\begin{align*}
  KU_s^\epsilon(\C{A}) = \pi_s(BU^\epsilon(\C{A})^+) & \xrightarrow{\ h\ } 
  H_s(BU^\epsilon(\C{A})^+) \xrightarrow{\ \sim\ } \\
  & H_s(BU^\epsilon(\C{A})) = H_s(U^\epsilon(\C{A}),k) \xrightarrow{\ L^\ell\ } HC^\a_{s+2\ell}(\C{A}),
\end{align*}
where $h$ is the Hurewicz homomorphism. In particular we have
\begin{equation}\label{DihedralChernCharacter}
{\rm Ch}_0^\ell:KU_0^\epsilon(\C{A})\xrightarrow{S_\ell} 
  HC^\a_{2\ell}(M_{2n}(\C{A})) \xrightarrow{{\rm Tr}_\ell} HC^\a_{2\ell}(\C{A}),
\end{equation}
where $S_\ell:[p] \mapsto [p\odots p]$.

\subsection{Hopf-dihedral homology of a group ring}

Following \cite{Misc76}, we note from \cite{GelfMisc69,Novi66} that
the problems of modifying even-dimensional multiply-connected
manifolds with fundamental group $\pi$ leads to a need to compute
$KU_0(\B{Z}\pi)$ for the group algebra $\B{Z}\pi$. As for the
odd-dimensional manifolds, it is shown in \cite{Novi70} that one
similarly needs to study $KU_1(\B{Z}\pi)$, where $\pi$ is the
fundamental group of the manifold.  See also \cite{Wall-book}.

Motivated by these discussions, we will study the dihedral Chern
character map \eqref{dihedral-chern} for the group algebra
$\C{A}=k\pi$ for a group $\pi$.  We will show that the dihedral Chern
character \eqref{dihedral-chern} lands in the Hopf-dihedral homology
of $k\pi$ which is a direct summand of the algebra dihedral homology
of $k\pi$. In view of Theorem \ref{thm-dihedral-cocomm} we gain a
computational advantage relating the $KU$-groups of $k\pi$ to the
group homology of the group $\pi$.

\begin{theorem}
  Let $\pi$ be a (discrete) group, and $\C{A}=k\pi$ the group algebra
  of the group $\pi$. Then the dihedral Chern character
  \eqref{dihedral-chern} lands in the Hopf-dihedral homology of the
  Hopf algebra $k\pi$.
\end{theorem}

\begin{proof}
  We first note, as a result of Proposition \ref{prop-group-dihedral},
  that we may replace, up to homology, the inclusion
  \begin{equation*}
    j_\bullet:HC_\bullet^\pm(U^\epsilon_n(\C{A}),k) \lra HC_\bullet^\pm(kU^\epsilon_n(\C{A})),
  \end{equation*}
  by the section
  \begin{equation*}
    \t_\bullet:HC_\bullet^\pm(kU^\epsilon_n(\C{A});1,\ve) \lra HC_\bullet^\pm(kU^\epsilon_n(\C{A}))
  \end{equation*}
  of the characteristic homomorphism
  $\gamma_\bullet:HC_\bullet^\pm(kU^\epsilon_n(\C{A})) \lra
  HC_\bullet^\pm(kU^\epsilon_n(\C{A});1,\ve)$.
  It then follows from Equation~\eqref{theta-group} that the image of
  $\t_\bullet$ is the $[1]$-component of the decomposition of the
  cyclic homology $HC_\bullet(kU^\epsilon_n(\C{A}))$ along the
  conjugacy classes of elements $[z]$ in $U^\epsilon_n(\C{A})$ 
  \[ \bigoplus_{[z]} HC_\bullet(k[B_\cdot
  ({U^\epsilon_n(\C{A})}_z,\,z)]) \]
  as described in \cite{Burg85} and \cite[Sect. 7.4]{Loday-book}. Now,
  the trace map induces a map on homology of the form
  \begin{equation*}
    Tr_\bullet\colon HC_\bullet(k[B_\cdot
    ({U^\epsilon_n(\C{A})}_1,\,1)]) \lra HC_\bullet(k[B_\cdot
    (\pi)_1],\, 1).
  \end{equation*}
  Then by \cite[Thm. I]{Burg85} we know that the $[1]$-component has
  homology of the form
  \begin{equation*}
    HC_\bullet(k[B_\cdot (\pi)_1,\,1)]) = \bigoplus_{i\geq 0}H_{\bullet - 2i}(\pi).
  \end{equation*}
  The result then follows.
\end{proof}

\subsection{Podleś spheres}

We next recall from \cite{NoumMima90} the (two parameters) quantum Podleś sphere $\C{O}(S_q(c,d))$, along with the coaction of the coordinate (Hopf) algebra $\C{O}(SU_q(2))$ of the quantum group $SU_q(2)$. 

The compact quantum group (CQG) algebra $\C{O}(SU_q(2))$ is the algebra generated by $x,u,v,y$ subject to the relations
\begin{align*}
& ux=qxu, \qquad vx=qxv, \qquad yu=quy, \qquad yv=qvy,\\
& vu=uv, \qquad xy-q^{-1}uv=yx-quv=1,
\end{align*}
and its Hopf algebra structure is given by
\begin{align*}
& \D(x)=x\ot x + u\ot v, \qquad \D(u)=x\ot u + u\ot y, \\
& \D(v)=v\ot x + y\ot v, \qquad \D(y)=v\ot u + y\ot y, \\
& \ve(x) = \ve(y) = 1, \qquad \ve(u) = \ve(v) = 0, \\
& S(x) = y, \quad S(y) = x, \quad S(u)=-qu, \quad S(v) = -q^{-1}v, \\
\end{align*}
see also \cite{MasuMimaNakaNoumUeno91}. We next note from \cite[Prop. 11.34]{KlimSchm-book} that there is a family $\{f_z\}_{z\in \B{C}}$ of characters of $\C{O}(SU_q(2))$, uniquely determined by 
\begin{enumerate}[(i)]
\item the functions $z\mapsto f_z(a)$ is an entire function of exponential growth on the right-half plane for any $a \in \C{O}(SU_q(2))$,
\item $f_zf_{z'}=f_{z+z'}$, with $f_0=\ve$,
\item $h(ab)=h(b(f_1\cdot a \cdot f_1))$, for any $a,b \in \C{O}(SU_q(2))$, and the Haar state $h$,
\end{enumerate}
satisfying\footnote[1]{It follows from the (non-degenerate) pairing between the CQG algebra $\C{O}(SU_q(2))$ and the QUE algebra $U_q(su_2)$, for the details of which we refer the reader to \cite{SchmWagn04}, that the character $f_1$ corresponds to the evaluation by $K^2 \in U_q(su_2)$.}
\begin{equation*}
S^2(a) = f_{-1} \cdot a \cdot f_1, \qquad f_z(S(a)) = f_{-z}(a)
\end{equation*}
for any $a\in \C{O}(SU_q(2))$. As a result, the pair $(f_1,1)$ is a MPI for the Hopf algebra $\C{O}(SU_q(2))$. Indeed,
\begin{equation*}
\widetilde{S}_1(a) = f_1(a_{(2)})S(a_{(1)}) = S(f_1\cdot a),
\end{equation*}
and therefore,
\begin{equation*}
\widetilde{S}^2_1(a) = f_1(a_{(3)})f_1(S(a_{(1)}))S^2(a_{(2)}) = f_1(a_{(3)})f_{-1}(a_{(1)})S^2(a_{(2)}) = S^2(f_1\cdot a \cdot f_{-1}) = a.
\end{equation*}

The Podleś spheres $\C{O}(S_q(c,d))$ on the other hand, is defined to be the algebra generated by $z_{-1}$, $z_0$, $z_1$ subject to the relations
\begin{align}\notag
z_0^2 -qz_1z_{-1}-q^{-1}z_{-1}z_1 = & d1, \\\label{relations-Podles}
(1-q^2)z_0^2+qz_{-1}z_1 - qz_1z_{-1} = & (1-q^2)cz_0, \\\notag
z_{-1}z_0 - q^2z_0z_{-1} = & (1-q^2)cz_{-1}, \\\notag
z_0z_1 - q^2z_1z_0 = & (1-q^2)cz_1.
\end{align}
In particular, the algebra $\C{O}(S_q(s,1+s^2))$ is denoted simply by $\C{O}(S^2_{qs})$, and it can be realized as a subalgebra of $\C{O}(SU_q(2))$ via $a_{i}\mapsto \widetilde{a_i}$ for $i=-1,0,1$ where
\begin{align*}
\widetilde{z_{-1}} := &(1+q^2)^{-1/2}x^2 + s(1+q^{-2})^{1/2}xv - q(1+q^2)^{-1/2}v^2, \\
\widetilde{z_0} := & ux + s(1+(q+q^{-1})uv) - vy, \\
\widetilde{z_1} := & (1+q^2)^{-1/2}u^2 + s(1+q^{-2})^{1/2}yu - q(1+q^2)^{-1/2}y^2.
\end{align*}
The algebra $\C{O}(S^2_{qs})$ carries the $\ast$-structure given by 
\begin{equation*}
{z_i}^* := (-q)^iz_{-i}.
\end{equation*} 
In case $s=0$, the algebra $\C{O}(S^2_{q0})$ is called the standard Podleś sphere, and is denoted by $\C{O}(S_q^2)$. Next, we recall the $\C{O}(SU_q(2))$-coaction from \cite[Prop. 4.25]{KlimSchm-book}. To this end, let
\begin{align*}
W_1 = & (w_{i,j})_{i,j\in \{-1,0,1\}} = \left(\begin{array}{ccc}
w_{-1,-1} & w_{-1,0} & w_{-1,1} \\
w_{0,-1} & w_{0,0} & w_{0,1} \\
w_{1,-1} & w_{1,0} & w_{1,1}
\end{array}\right) \\
= & \left(\begin{array}{ccc}
x^2 & (1+q^2)^{1/2}xu & u^2 \\
(1+q^2)^{1/2}xv & 1+(q+q^{-1})uv & (1+q^2)^{1/2}uy \\
v^2 & (1+q^2)^{1/2}vy & y^2
\end{array}\right).
\end{align*}
Then, the comultiplication on $\C{O}(SU_q(2))$ induces (once $a_i$'s are identified with $\widetilde{a}_i$'s) a coaction of the form $\nb\colon\C{O}(S_{qs}^2) \lra \C{O}(S_{qs}^2) \ot \C{O}(SU_q(2))$,
\begin{equation*}
\nb(z_i)= z_j \ot w_{j,i},
\end{equation*}
where $i,j\in\{-1,\,0,\,1\}$.  A left version of this coaction can be
found in \cite{MasuNakaWata91}.

Let us next recall from \cite{SchmWagn04}, and from \cite{Podl87,BrzeHaja09}, that the Podleś sphere is a (left) coideal subalgebra, as such, taking the quotient by the Hopf ideal $I={\C{O}(S^2_{qs})}^+\C{O}(SU_q(2))$, we obtain the $\C{O}(U(1))$-comodule algebra structure
\begin{equation*}
\xymatrix{
\C{O}(S^2_{qs}) \ar[rr]^{\rho}\ar[rd]_{\nb} & & \C{O}(S^2_{qs}) \ot \C{O}(U(1)) \\
& \C{O}(S_{qs}^2) \ot \C{O}(SU_q(2))  \ar[ru]_{\Id\ot \pi_I} &
}
\end{equation*}
where $\C{O}(U(1))$ is the Hopf algebra (of regular functions on the circle) generated by two group-likes $\s, \s^{-1}$. Explicitly, $\rho: \C{O}(S^2_{qs}) \lra \C{O}(S^2_{qs}) \ot \C{O}(U(1))$ is defined as
\begin{align*}
\rho(z_{-1}) = & z_{-1} \ot \s^2,\\
\rho(z_0) = & z_0 \ot 1,\\
\rho(z_1) = & z_1 \ot \s^{-2}.
\end{align*}
We further note from \cite[Sect. 4.5]{KlimSchm-book} that setting
\begin{equation*}
\C{O}(S^2_{qs})[n]:=\{a\in \C{O}(S^2_{qs}) \mid \rho(a) = a \ot \s^n\},
\end{equation*}
we have $z_{-1}^iz_0^jz_1^k\in \C{O}(S^2_{qs})[2i-2k]$, where $i,j,k \geq 0$, that
\begin{equation}\label{grade-Podles}
\C{O}(S^2_{qs}) = \bigoplus_{m\in \B{Z}}\C{O}(S^2_{qs})[2m],
\end{equation}
and that the Haar state $h$ vanishes on $\C{O}(S^2_{qs})[2m]$ for $m\neq 0$.

\begin{lemma}
The Haar state $h$ of the CQG algebra $\C{O}(SU_q(2))$ induces a $1$-invariant $f_1$-trace on the algebra $\C{O}(S^2_q)$.
\end{lemma}

\begin{proof}
For any $a,b \in \C{O}(SU_q(2))$, it follows from the modular property of the Haar functional that
\begin{equation*}
h(ab) = h(b(f_1\cdot a \cdot f_1)) = h(b(a \cdot f_1)) = h(b(f_1\cdot a_{(2)}))f_1(a_{(1)}) = h(b(f_1\cdot a)) = h(ba_{(1)})f_1(a_{(2)}),
\end{equation*}
using, on the forth equality, the fact that $f_1(a_{(1)})a_{(2)} = a$ for any $a\in \C{O}(S^2_{qs})$, see for instance \cite[Eq. (22)]{SchmWagn04}.
Next, by the invariance property \cite[Def. 4.3]{KlimSchm-book}, see also \cite{NoumMima90} we readily have
\begin{equation*}
h(a_{(1)})a_{(2)} = a_{(1)}h(a_{(2)}) = h(a)1,
\end{equation*}
for any $a\in \C{O}(SU_q(2))$.
\end{proof}

As a result, we have the characteristic homomorphism $\gamma_n\colon
\CC_n(\C{O}(S^2_{qs})) \lra \CC_n(\C{O}(U(1));1,\ve)$ defined as
\begin{align}\label{char-map-Sq}
\gamma_n(a_0 \odots a_n) := h(a_0{a_1}\ns{0}\ldots {a_n}\ns{0}){a_1}\ns{1}\odots {a_n}\ns{1}
\end{align}
to transfer the $KU$-classes of $\C{O}(S^2_{qs})$ to the Hopf-dihedral homology of the (cocommutative) Hopf algebra $\C{O}(U(1))$.

By means of the Hopf-dihedral Chern character
\begin{equation*}
{\rm Ch}_n^\pm:KU_n(\C{O}(S^2_{qs})) \lra HC^\pm_n(\C{O}(U(1));1,\ve).
\end{equation*}
we have the Hopf-dihedral homology of $\C{O}(U(1))$ to classify the $KU$-classes of $\C{O}(S^2_{qs})$.

We consider the composition
$$H_m(U_n^\epsilon(\C{O}(S^2_{qs})),k) \lra
HC^\a_{m+2\ell}(U_n^\epsilon(\C{O}(S^2_{qs})),k) \lra
HC^\a_{m+2\ell}(\C{O}(U(1)),k),$$
where $\a=(-1)^\ell$. Since
\begin{align*}
& H_0(\C{O}(S^1),k) \cong H_0(\B{Z},k) = k, \\
& H_1(\C{O}(S^1),k) \cong H_1(\B{Z},k) = \langle \s \rangle \cong k, \\
& H_m(\C{O}(S^1),k) \cong H_m(\B{Z},k) = 0, \quad m\geq 2,
\end{align*}
for the details of which we refer the reader to \cite{Brown-book}, we have
\begin{equation*}
H_0(U_n^\epsilon(\C{O}(S^2_{qs})),k) = k_{U_n^\epsilon(\C{O}(S^2_{qs}))} = k = k_{\B{Z}} = H_0(\B{Z},k).
\end{equation*}
As for $H_1(U_n^\epsilon(\C{O}(S^2_{qs})),k)$, we first recall that since $BU^\epsilon(\C{O}(S^2_{qs}))^+$ is path connected, the Hurewicz map 
\begin{align*}
h\colon KU_1^\epsilon(\C{O}(S^2_{qs})) & = \pi_1(BU^\epsilon(\C{O}(S^2_{qs}))^+) = \frac{U^\epsilon(\C{O}(S^2_{qs}))}{[U^\epsilon(\C{O}(S^2_{qs})),\,U^\epsilon(\C{O}(S^2_{qs}))]} \\
\lra & \frac{\pi_1(BU^\epsilon(\C{O}(S^2_{qs}))^+)}{[\pi_1(BU^\epsilon(\C{O}(S^2_{qs}))^+),\,\pi_1(BU^\epsilon(\C{O}(S^2_{qs}))^+)]} \cong H_1(U_n^\epsilon(\C{O}(S^2_{qs})),k)
\end{align*}
is the canonical abelianization, and therefore in this case is
identity. On the other hand, it follows from the relations \eqref{relations-Podles} that the element
\begin{equation*}
u =\left(\begin{array}{cc}
q^{-1}z_0 & \sqrt{1+q^{-2}}z_1 \\
\sqrt{1+q^{2}}z_{-1} & qz_0
\end{array}\right)
\end{equation*}
is a unitary matrix on $\C{O}(S_q^2)$. Hence, for any
\begin{equation*}
\left(\begin{array}{cc}
\a & \b \\
\gamma & \t
\end{array}\right) \in U^\epsilon_{n}(\B{R})
\end{equation*}
we have
\begin{equation*}
M = \left(\begin{array}{cc}
\a\ot u & \b\ot u \\
\gamma\ot u & \t\ot u
\end{array}\right) \in U_{2n}^{\epsilon}(\C{O}(S^2_{q})).
\end{equation*}
For simplicity, we may take 
\begin{equation*}
M = \left(\begin{array}{cc}
0 & \epsilon I_n\ot u \\
\epsilon I_n \ot u & 0
\end{array}\right) \in U_{2n}^{\epsilon}(\C{O}(S^2_{q})).
\end{equation*}
Then,
\begin{align*}
[M] \mapsto & \left[\sum_{1 \leq i_0, i_1\leq 2n} h({{M^{-1}}_{i_0i_1}}{{{M}}_{i_1i_0}}\ns{0}){{{M}}_{i_1i_0}}\ns{1}\right] = \left[\sum_{1 \leq i_0, i_1\leq 2n} n h({{u^\dagger}_{i_0i_1}}{{u}_{i_1i_0}}\ns{0}){{u}_{i_1i_0}}\ns{1}\right] \\
    = & 2 n((1 + q^{-2})h(z_1^*z_1) - (1 + q^2)h(z_{-1}^*z_{-1}))[\s].
\end{align*}
This proves that an element $[M]\in KU^+_0(\C{O}(S^2_q))$ is sent to
a non-trivial element in $HC^+_1(\C{O}(U(1));1,\varepsilon)$ under the
dihedral Chern character. Thus we proved
\begin{theorem}
  $KU^+_1(\C{O}(S^2_q))$ is non-trivial.
\end{theorem}

\bibliographystyle{plain}
\bibliography{references}{}

\end{document}